\documentclass[a4paper]{amsart}
\usepackage[ps2pdf]{hyperref}
\usepackage{amsmath,amssymb}
\usepackage{bbold}
\usepackage{psfrag}
\usepackage{graphicx}
\usepackage{leftidx}
\usepackage[all,poly]{xy}
\usepackage[latin1]{inputenc}

\newtheorem{thm}{Theorem}[section]
\newtheorem{cor}[thm]{Corollary}
\newtheorem{lem}[thm]{Lemma}

\newcommand{\rhom}[1]{[#1]^r}
\newcommand{\lhom}[1]{[#1]^l}
\newcommand{\leva}[2]{\mathrm{ev}^{#1}_{#2}}
\newcommand{\reva}[2]{\widetilde{\mathrm{ev}}^{#1}_{#2}}
\newcommand{\lcoeva}[2]{\mathrm{coev}^{#1}_{#2}}
\newcommand{\rcoeva}[2]{\widetilde{\mathrm{coev}}^{#1}_{#2}}
\newcommand{\Set}{\mathbf{Set}}
\newcommand{\Cat}{\mathbf{Cat}}

\newcommand{\reve}{\mathrm{rev}}

\newcommand{\CM}{\chi}

\newcommand{\co}{\colon}
\newcommand{\id}{\mathrm{id}}

\newcommand{\conv}{\mathrm{Conv}}
\newcommand{\opp}{\mathrm{op}}

\newcommand{\cc}{\mathcal{C}}
\newcommand{\tg}{\mathcal{G}}
\newcommand{\cch}{\mathcal{C}_{\mathrm{hom}}}
\newcommand{\bb}{\mathcal{B}}
\newcommand{\dd}{\mathcal{D}}

\newcommand{\ee}{\mathcal{E}}

\newcommand{\ff}{\mathcal{F}}

\newcommand{\sss}{\mathcal{S}}

\newcommand{\alg}{\mathrm{Alg}}

\newcommand{\Mod}{\mathrm{Mod}}
\newcommand{\Ker}{\mathrm{Ker}}
\newcommand{\Coker}{\mathrm{Coker}}

\newcommand{\comod}{\mathrm{comod}}
\newcommand{\Comod}{\mathrm{Comod}}

\newcommand{\rmod}[1]{\mathrm{Mod}_{#1}}
\newcommand{\Hmod}[1]{\leftidx{_{#1}^{#1}}{\mathcal{H}}{}}

\newcommand{\kk}{\Bbbk}
\newcommand{\lact}[2]{\leftidx{^{#1}}{{#2}}{}}
\newcommand{\elact}[2]{\leftidx{^{#1}}{{\hspace{-.08em} #2}}{}}
\newcommand{\kt}{$\Bbbk$\nobreakdash-\hspace{0pt}}

\newcommand{\ti}{\mbox{-}\,}

\newcommand{\un}{\mathbb{1}}

\newcommand{\Ob}{\mathrm{Ob}}

\newcommand{\Aut}{\mathrm{Aut}}

\newcommand{\End}{\mathrm{End}}
\newcommand{\Hom}{\mathrm{Hom}}

\newcommand{\Ima}{{\mathrm{Im}}}

\newcommand{\mmod}{\mathrm{mod}}

\newcommand{\lev}{\mathrm{ev}}

\newcommand{\rev}{\widetilde{\mathrm{ev}}}
\newcommand{\lcoev}{\mathrm{coev}}
\newcommand{\rcoev}{\widetilde{\mathrm{coev}}}
\newcommand{\ldual}[1]{\leftidx{^\vee}{\!#1}{}}
\newcommand{\rdual}[1]{{#1}^\vee}

\newcommand{\scaledraw}[1]{A}
\newcommand{\scaleraisedraw}[2]{A}
\newcommand{\rsdraw}[3]{\raisebox{-#1\height}{\scalebox{#2}{\includegraphics{#3.eps}}}}
\newcommand{\labela}{\renewcommand{\labelenumi}{{\rm (\alph{enumi})}}}

\begin{document}

\title[Hopf crossed module (co)algebras]{Hopf crossed module (co)algebras}
\author{K\"{u}r\c{s}at S\"{o}zer}
 \address{K\"{u}r\c{s}at S\"{o}zer\newline
  \indent  McMaster University, Hamilton, ON L8S 4E8, Canada\\}
\email{sozerk@mcmaster.ca}
\author{Alexis Virelizier}
\address{Alexis Virelizier\newline
\indent Univ. Lille, CNRS, UMR 8524 - Laboratoire Paul Painlev\'e, F-59000 Lille, France\\}
\email{alexis.virelizier@univ-lille.fr}

\date{\today}

\begin{abstract}
Given a crossed module $\chi$, we introduce Hopf $\chi$-(co)algebras which generalize Hopf algebras and Hopf group-(co)algebras. We interpret them as Hopf algebras in some symmetric monoidal category. We prove that their categories of representations are monoidal and  $\chi$-graded  (meaning that both objects and morphisms have degrees which are related via $\CM$).
\end{abstract}

\maketitle

\setcounter{tocdepth}{1}
\tableofcontents

\section{Introduction}\label{sec-Intro}

Hopf algebras, such as quantum groups, are fundamental objects in the field of quantum algebra and quantum topology. In particular, Hopf algebras and their categories of representations are very useful in the construction of quantum invariants of knots and 3-manifolds (see for instance \cite{Jo,RT,TV,Ko,Ku,He,LMO}).

There are various generalizations of Hopf algebras. A particular generalization is given by so-called Hopf group-coalgebras introduced by Turaev for topological purposes:
given a (discrete) group~$G$, Hopf $G$-coalgebras and their categories of representations (which are $G$-graded monoidal categories) are used to define quantum invariants of principal $G$-bundles over 3-manifolds (see \cite{Tu1}).   On the algebraic side, the second author generalized to Hopf group-coalgebras most of the classical results for Hopf algebras (see~\cite{Vi2}), and  Caenepeel and De Lombaerde showed that Hopf group-coalgebras are Hopf algebra objects in a certain symmetric monoidal category (see~\cite{CD}).

In this paper, we introduce and study Hopf crossed module-coalgebras which are extensions of Hopf group-coalgebras. Recall from homotopical algebra that crossed modules are a convenient way of encoding (strict) 2-groups. Explicitly, a crossed module is a group homomorphism $\CM \co E \to H$ together with an action of $H$ on $E$ such that~$\CM$ is $H$-equivariant and satisfies the Peiffer identity (see Section~\ref{sect-crossed-modules-def}).

Given a crossed module $\CM \co E \to H$, a Hopf $\CM$-coalgebra is a Hopf $H$-coalgebra endowed with a $\CM$-action (see Section \ref{sect-Hopf-crossed-module-coalgebras}). More explicitly, a Hopf $\CM$-coalgebra (over a commutative ring~$\kk$) is a family $A=\{A_x\}_{x \in H}$ of $\kk$-algebras endowed with a comultiplication $\Delta = \{ \Delta_{x,y}\co A_{xy} \to A_x \otimes A_y\}_{x,y \in H}$, a counit $\varepsilon \co A_1 \to \kk$, an antipode $S=\{ S_x \co A_{x^{-1}} \to A_x\}_{x \in H}$, and a $\CM$-action consisting of a family of algebra isomorphisms
$$
\phi = \{ \phi_{x,e} \co A_x \to A_{\CM(e)x} \}_{(x,e) \in H \times E}.
$$
These data should verify some compatibility conditions generalizing the axioms of a Hopf algebra and of an action of a group. There are two particular cases. First, if $H$ is a group, then the trivial map $1 \to H$ is a crossed module and a Hopf $(1 \to H)$-coalgebra is a Hopf $H$-coalgebra. Second, if $E$ is an abelian group, then the map $E \to 1$ is a crossed module and a Hopf $(E \to 1)$-coalgebra is a Hopf algebra endowed with an action of $E$ by algebra and bicomodule automorphisms. 

Our main motivation for introducing Hopf $\CM$-coalgebras  is to produce instances of  monoidal categories which are $\CM$-graded (in the sense of \cite{SV}). In such a category, not only the objects have a degree in $H$, but also the morphisms have a degree in $E$, and these two degrees are related by the crossed module homomorphism~$\CM$. Actually, $\CM$-graded monoidal categories are useful for topological purposes: it is shown in \cite{SV} that any $\CM$-graded $\CM$-fusion category gives rise to a state sum invariant 
of  3-manifolds endowed with a homotopy class of maps to the classifying space $B\CM$ of $\CM$ (which is a homotopy 2-type) and, more generally, to a 3-dimensional homotopy quantum field theory with target~$B\CM$.

The first main result of the paper is that the category $\Mod_\CM(A)$ of modules over a Hopf $\CM$-coalgebra $A$ is a $\CM$-graded monoidal category with internal Homs (see Theorem~\ref{thm-mod-Xi-A}). We also study its full subcategory $\mmod_\CM(A)$ whose objects have their underlying module projective of finite rank. We prove that the pivotal structures on $\mmod_\CM(A)$ are in bijective correspondence with the pivotal elements of~$A$ (see Corollary~\ref{cor-mmod}), and we provide sufficient conditions on $A$ for $\mmod_\CM(A)$ to be a $\CM$-fusion category (see Theorem~\ref{thm-Xi-fusion}).

Note that the notion of a Hopf $\CM$-coalgebra is not self-dual: the dual notion is that of a Hopf $\CM$-algebra (see Section \ref{sect-dual-notion}) and the category $\Comod_\CM(A)$ of comodules over a Hopf $\CM$-algebra $A$ is a closed $\CM$-graded monoidal category (see Section~\ref{sect-dual-A-comodules}). 

Next, we introduce the notion of a Hopf $\CM$-module over a Hopf $\CM$-coalgebra and prove a structure theorem for them (see Theorem~\ref{thm-structure-Hopf-Xi-modules}). When the ground ring is a field and the Hopf $\CM$-coalgebra is of finite type, we derive from this structure theorem the existence and uniqueness of $\CM$-integrals (see Theorem~\ref{thm-existence_integrals}). These  generalize the well known corresponding results for Hopf algebras.

Finally, we interpret  Hopf crossed module-(co)algebras in any symmetric monoi\-dal category $\sss$ as Hopf algebra objects in some  symmetric monoidal category associated with $\sss$ (see Theorems~\ref{thm-characterization-Hopf-Xi-coalgebras-in-S} and \ref{thm-characterization-Hopf-Xi-algebras-in-S}). In particular,  the case where $\sss$ is the category $\Mod_\kk$ of \kt modules corresponds to the Hopf crossed module-(co)algebras over $\kk$ (considered above). This is built on the fact that crossed modules are group objects in the category of small categories (see \cite{BS}) and generalizes the above cited work of Caenepeel and De Lombaerde.

The paper is organized as follows. In Section~\ref{sect-2}, we review monoidal categories and the associated graphical calculus. In Section~\ref{sect-3}, we discuss the notions of cate\-gorical (co)algebras, Hopf algebras, and (co)modules. We recall crossed modules in Section~\ref{sect-4} and crossed module graded categories in Section~\ref{sect-5}. In Section~\ref{sect-Hopf-group-coalgebra}, we discuss Hopf group-coalgebras and their categories of representations.
In Section~\ref{sect-Hopf-gcrossed-modules-coalgebra}, we introduce   Hopf crossed module-(co)algebras. Section~\ref{sect-representations-Hopf-Xi-coalgebras} is devoted to the study of their categories of  representations. In Section~\ref{sect-Hopf-modules-integrals}, we introduce and characterise Hopf crossed module-modules and study integrals of Hopf crossed module-coalgebras.  Finally, in Section~\ref{sect-Hopf-Xi-as-Hopf}, we interpret Hopf crossed module-(co)algebras as Hopf algebras in some symmetric monoidal category.\\

Throughout the paper, we fix a nonzero commutative ring~$\kk$. The tensor product over $\kk$ is denoted $\otimes_\Bbbk$ or more simply $\otimes$ if there is no confusion.

\section{Categorical preliminaries}\label{sect-2}

\subsection{Conventions on monoidal categories}\label{sect-conventions}
For the basics on monoidal categories, we refer for example to \cite{ML1,EGNO,TVi5}. We  will suppress in our formulas  the associativity  and unitality  constraints   of monoidal categories.   This does not lead  to   ambiguity because by   Mac Lane's   coherence theorem,   all legitimate ways of inserting these constraints   give the same result.
For  any objects $X_1,...,X_n$ with $n\geq 2$, we set
$$
X_1 \otimes X_2 \otimes \cdots \otimes X_n=(... ((X_1\otimes X_2) \otimes X_3) \otimes
\cdots \otimes X_{n-1}) \otimes X_n
$$
and similarly for morphisms.

\subsection{Braided and symmetric categories}\label{sect-braided-categories}
A \emph{braiding} on a monoidal category $\bb=(\bb,\otimes,\un)$ is a natural isomorphism
$
\tau=\{\tau_{X,Y} \co  X \otimes Y\to Y \otimes X\}_{X,Y \in\bb}
$
such that for all objects $X,Y,Z \in \bb$,
\begin{equation*}
\tau_{X, Y\otimes Z}=(\id_Y \otimes \tau_{X, Z})(\tau_{X, Y} \otimes \id_Z) \quad \text{and} \quad
\tau_{X\otimes Y,Z}=(\tau_{X,Z} \otimes \id_Y)(\id_X \otimes \tau_{Y,Z}).
\end{equation*}
A \emph{braided category} is a monoidal category endowed with a braiding.

A \emph{braiding} on a monoidal category $\bb$ is \emph{symmetric} if $\tau_{Y,X} \tau_{X,Y}=\id_{X \otimes Y}$ for all $X,Y\in \bb$.
A \emph{symmetry} is a symmetric braiding. A \emph{symmetric monoidal category} is a monoidal category endowed with a  symmetry.

\subsection{Cartesian monoidal categories}\label{sect-cartesian-monoidal-categories}
A \emph{cartesian monoidal category} is a monoi\-dal category whose monoidal structure is given by the category-theoretic product (and so whose unit object is a terminal object). Such a category is then symmetric, with symmetry given by the canonical flip maps.

Any category with finite products can be considered as a cartesian monoidal category (as long as we have a specified product for each pair of objects). In particular, the category $\Set$ of sets and maps, and the category $\Cat$ of small categories and functors, endowed with their canonical category-theoretic product, are  symmetric monoidal cartesian categories.

For any object $X$ of a cartesian monoidal category, there is a unique morphism $\Delta_X \co X \to X \otimes X$, called the \emph{diagonal map}, such that $\Delta_X$ composed with the first or second projection is the identity and, since the unit  object is a terminal object,  there is a unique morphism $\varepsilon_X \co X \to \un$, called the \emph{augmentation}.

\subsection{Rigid categories}\label{sect-rigid}
A \emph{duality} in a monoidal category $\cc$ is a quadruple $(X,Y,e,d)$, where $X$, $Y$ are objects of
$\cc$, $e\co X \otimes Y \to \un$ (the \emph{evaluation}) and $d \co \un \to Y \otimes X$ (the \emph{coevaluation}) are
morphisms in $\cc$, such that
\begin{equation*}
(e \otimes \id_X)(\id_X \otimes d)=\id_X \quad \text{and} \quad (\id_Y \otimes e)(d \otimes \id_Y)=\id_Y.
\end{equation*}
Then $(X,e,d)$ is a \emph{left dual} of $Y$ and $(Y,e,d)$ is a \emph{right dual} of $X$.

Left and right duals, if they exist, are essentially unique: if $(Y,e,d)$ and $(Y',e',d')$ are right duals of some
object $X$, then there exists a unique isomorphism $u\co Y \to Y'$ such that $e'=e(\id_X \otimes u^{-1})$ and $d'=(u
\otimes \id_X)d$.

A monoidal category is \emph{left rigid} (respectively, \emph{right rigid}) if every object admits a left dual  (respectively, a right dual). A \emph{rigid category} is a monoidal category  which is both left rigid and right rigid.

Subsequently, when dealing with rigid categories, we shall always assume tacitly that for each object
$X$, a left dual $(\ldual{X},\lev_X,\lcoev_X)$ and a right dual $(\rdual{X},\rev_X,\rcoev_X)$ has been chosen.
Such a choice defines a left  dual functor $\ldual{?}\co \cc^\reve\to \cc$ and a right  dual functor $\rdual{?}\co  \cc^\reve \to \cc$,  where $\cc^\reve =(\cc^\opp,\otimes^\opp,\un)$ is the opposite category to~$\cc$ with opposite monoidal structure. In particular, the left and right  duals of any morphism $f\co X \to Y$ in $\cc$  are defined by
\begin{align*}
\ldual{f}&= (\lev_Y \otimes  \id_{\ldual{X}})(\id_{\ldual{Y}}  \otimes f \otimes \id_{\ldual{X}})(\id_{\ldual{Y}}\otimes \lcoev_X) \co \ldual{Y} \to \ldual{X},\\
\rdual{f}&= (\id_{\rdual{X}} \otimes \rev_Y)(\id_{\rdual{X}} \otimes f \otimes \id_{\rdual{Y}})(\rcoev_X \otimes \id_{\rdual{Y}})\co \rdual{Y} \to \rdual{X}.
\end{align*}
The left and right dual functors are monoidal. Note that the actual choice of left and right duals is innocuous in the sense that different choices of left (respectively,\@ right) duals define canonically monoidally isomorphic left (respectively,\@ right) dual
functors.

\subsection{Closed monoidal categories}\label{sect-closed}
A monoidal category $\cc$ is \emph{left closed} if for any object
$X$ of $\cc$, the endofunctor $? \otimes X$ has a right adjoint $\lhom{X,?}$, with adjunction unit and counit:
$$
\leva{X}{Y} \co  \lhom{X,Y}\otimes  X\to Y \quad\text{and}\quad \lcoeva{X}{Y} \co Y \to \lhom{X, Y\otimes X},
$$
called respectively the \emph{left evaluation} and the \emph{left coevaluation}. Then $\lhom{X,Y}$ is called the \emph{left internal Hom} from $X$ to $Y$. The left internal Homs give rise to a functor $\lhom{-,-}\co \cc^\opp \times \cc \to \cc$.

Similarly, a monoidal category $\cc$ is \emph{right closed} if for any object
$X$ of $\cc$, the endofunctor $X \otimes ?$ has a right adjoint $\rhom{X,?}$, with adjunction unit and counit:
$$
\reva{X}{Y} \co  X \otimes \rhom{X,Y}\to Y \quad\text{and}\quad\rcoeva{X}{Y} \co Y \to \rhom{X, X\otimes Y},
$$
called respectively the \emph{right evaluation} and the \emph{right coevaluation}. Then $\rhom{X,Y}$ is called the \emph{right internal Hom} from $X$ to $Y$. The right internal Homs give rise to a functor $\rhom{-,-}\co \cc^\opp \times \cc \to \cc$.

A monoidal category is \emph{closed} if it is both left and right closed.
For example, the category $\Mod_\kk$ of \kt modules and \kt linear homomorphisms is closed: the internal Homs between \kt modules $M$ and $N$ are
$
\lhom{M,N}=\rhom{M,N}=\Hom_\kk(M,N)
$
with the standard (co)evaluations.

Any rigid monoidal category is closed: the internal Homs
are $\lhom{X,Y}= Y \otimes \ldual{X}$ and $\rhom{X,Y}= \rdual{X}\otimes Y$
with (co)evaluations
\begin{align*}
\leva{X}{Y}&=\id_Y \otimes \lev_X,  & \lcoeva{X}{Y}&=\id_Y \otimes \lcoev_X, \\
\reva{X}{Y}&=\rev_X \otimes \id_Y, & \rcoeva{X}{Y}&=\rcoev_X \otimes \id_Y.
\end{align*}

\subsection{Pivotal categories}\label{sec-pivot}
A \emph{pivotal category} is a rigid category $\cc$  endowed with a monoidal isomorphism between the left and the right dual functors. By
modifying the right duals of objects using this monoidal isomorphism, we may assume it to be the identity without loss of generality. In other words, for each
object $X$ of $\cc$, we have a \emph{dual object}~$X^*$ and four morphisms
\begin{align*}
& \lev_X \co X^*\otimes X \to\un,  \qquad \lcoev_X\co \un  \to X \otimes X^*,\\
&   \rev_X \co X\otimes X^* \to\un, \qquad   \rcoev_X\co \un  \to X^* \otimes X,
\end{align*}
such that $(X^*,\lev_X,\lcoev_X)$ is a left
dual for $X$, $(X^*, \rev_X,\rcoev_X)$ is a right
dual for~$X$, and the induced left and right dual functors coincide as monoidal functors.

For example, any left rigid symmetric monoidal category has a canonical structure of a pivotal category, for which the right (co)evaluations are given by the left ones (pre)composed with the symmetry. In particular, the full subcategory  of  $\Mod_\kk$ consisting of projective \kt modules of finite rank has a canonical structure of a
pivotal category: the dual of  projective \kt module $M$ of finite rank is $M^*=\Hom_\kk(M,\kk)$ with the standard (co)evaluations.

\subsection{Penrose graphical calculus}\label{sect-Penrose}
Morphisms in a monoidal category may  be represented by planar diagrams to be read from bottom to top.  We discuss here the basics of this Penrose graphical calculus (see \cite{JS} or~\cite{TVi5} for a detailed treatment). The  diagrams are made up of arcs colored by objects and of boxes colored by morphisms.  The arcs connect
the boxes and   have no   intersections or self-intersections.
The identity $\id_X$ of an object $X$, a morphism $f\co X \to Y$,
and the composition of two morphisms $f\co X \to Y$ and $g\co Y \to
Z$ are represented as follows:
$$
\psfrag{X}[Bl][Bl]{\scalebox{.8}{$X$}} \psfrag{Y}[Bc][Bc]{\scalebox{.8}{$Y$}} \psfrag{h}[Bc][Bc]{\scalebox{.9}{$f$}} \psfrag{g}[Bc][Bc]{\scalebox{.9}{$g$}}
\psfrag{Z}[Bc][Bc]{\scalebox{.8}{$Z$}} \id_X=\, \rsdraw{.45}{.9}{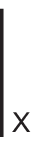}\,,\quad f=\, \rsdraw{.45}{.9}{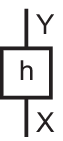}\;\,,\quad \text{and} \quad gf=\, \rsdraw{.45}{.9}{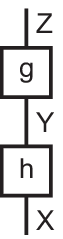}\,\; .
$$
The monoidal product of two morphisms $f\colon X \to Y$
and $g \colon U \to V$ is represented by juxtaposition:
$$
\psfrag{X}[Bl][Bl]{\scalebox{.8}{$X$}}
\psfrag{h}[Bc][Bc]{\scalebox{.9}{$f$}}
\psfrag{Y}[Bl][Bl]{\scalebox{.8}{$Y$}}
f\otimes g= \,\rsdraw{.45}{.9}{morphism-bt}\,\;
\psfrag{X}[Bl][Bl]{\scalebox{.8}{$U$}}
\psfrag{g}[Bc][Bc]{\scalebox{.9}{$g$}}
\psfrag{Y}[Bl][Bl]{\scalebox{.8}{$V$}}
\rsdraw{.45}{.9}{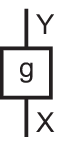}\,\;   .
$$
We can also use boxes with several strands attached to their horizontal sides. For example, a morphism $f \colon X \otimes Y \to A \otimes B \otimes C$ may be represented in various ways,   such   as
$$
\psfrag{X}[Bl][Bl]{\scalebox{.8}{$X$}}
\psfrag{h}[Bc][Bc]{\scalebox{.9}{$f$}}
\psfrag{Y}[Bl][Bl]{\scalebox{.8}{$Y$}}
\psfrag{A}[Bl][Bl]{\scalebox{.8}{$A$}}
\psfrag{B}[Bl][Bl]{\scalebox{.8}{$B$}}
\psfrag{C}[Bl][Bl]{\scalebox{.8}{$C$}}
\rsdraw{.45}{.9}{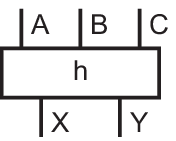}
\qquad \text{or} \qquad
\psfrag{X}[Bl][Bl]{\scalebox{.8}{$X \otimes Y$}}
\psfrag{h}[Bc][Bc]{\scalebox{.9}{$f$}}
\psfrag{A}[Bl][Bl]{\scalebox{.8}{$A$}}
\psfrag{B}[Bl][Bl]{\scalebox{.8}{$B \otimes C$}}
\rsdraw{.45}{.9}{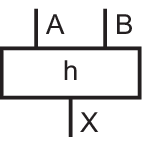}
\qquad \text{or} \qquad
\psfrag{X}[Bl][Bl]{\scalebox{.8}{$X$}}
\psfrag{h}[Bc][Bc]{\scalebox{.9}{$f$}}
\psfrag{Y}[Bl][Bl]{\scalebox{.8}{$Y$}}
\psfrag{A}[Br][Br]{\scalebox{.8}{$A \otimes B$}}
\psfrag{B}[Bl][Bl]{\scalebox{.8}{$C$}}
\rsdraw{.45}{.9}{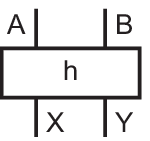} \,\;.
$$
Here, in accordance with the conventions of  Section~\ref{sect-conventions}, we  ignore here the associativity   constraint   between the  objects $A \otimes B \otimes C =(A \otimes B)\otimes C$ and $A \otimes (B\otimes C)$.
A~box whose lower/upper side has no  attached strands represents a morphism with source/target $\un$. For example,
morphisms $\alpha\colon \un \to \un$,   $\beta\colon \un \to X$,   $\gamma\colon X \to \un$  may be represented by the diagrams
$$
\psfrag{h}[Bc][Bc]{\scalebox{.9}{$\alpha$}}
\rsdraw{.45}{.9}{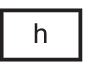}\;, \qquad
\psfrag{X}[Bl][Bl]{\scalebox{.8}{$X$}}
\psfrag{h}[Bc][Bc]{\scalebox{.9}{$\beta$}}
\psfrag{Y}[Bl][Bl]{\scalebox{.8}{$Y$}}
\rsdraw{.45}{.9}{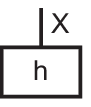}\;, \qquad
\psfrag{X}[Bl][Bl]{\scalebox{.8}{$X$}}
\psfrag{h}[Bc][Bc]{\scalebox{.9}{$\gamma$}}
\psfrag{Y}[Bl][Bl]{\scalebox{.8}{$Y$}}
\rsdraw{.45}{.9}{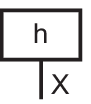}\;.
$$

Every diagram which is colored as above determines a morphism obtained as follows. First slice the diagram into horizontal strips so that each strip is made of juxtaposition of vertical segments or boxes. Then, for each strip, take the monoidal product of the morphisms associated to the vertical segments or boxes. Finally,   compose the resulting morphisms  proceeding from the bottom to the top. For example, given morphisms $f\colon Y \to Z$, $g\colon B \otimes Z \to \un$, $h \colon X \to A \otimes B$,  the diagram
$$
\psfrag{h}[Bc][Bc]{\scalebox{.9}{$g$}}
\psfrag{k}[Bc][Bc]{\scalebox{.9}{$f$}}
\psfrag{b}[Bc][Bc]{\scalebox{.9}{$h$}}
\psfrag{X}[Bl][Bl]{\scalebox{.8}{$X$}}
\psfrag{Y}[Bl][Bl]{\scalebox{.8}{$Z$}}
\psfrag{A}[Bl][Bl]{\scalebox{.8}{$A$}}
\psfrag{B}[Bl][Bl]{\scalebox{.8}{$B$}}
\psfrag{Z}[Bl][Bl]{\scalebox{.8}{$Y$}}
\rsdraw{.45}{.9}{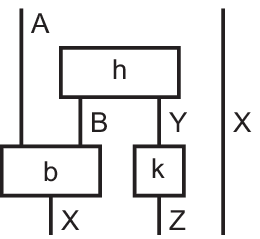}
$$
represents the morphism
$$
(\id_A \otimes g \otimes \id_X)(h \otimes f \otimes \id_X)=\bigl ( (\id_A \otimes g)(h \otimes f) \bigr ) \otimes \id_X
$$
from $X \otimes Y \otimes X$  to $A \otimes X$.

The functoriality of the monoidal product implies that the morphism associated to a colored diagram is  independent of the way of cutting it into horizontal strips. It also implies that we can push boxes lying on  the same horizontal level  up or down
without changing the morphism represented by the diagram. For example, for all morphisms $f\colon X \to Y$ and $g\colon U \to
V$ in~$\sss$, we have:
$$
\psfrag{X}[Bl][Bl]{\scalebox{.8}{$X$}}
\psfrag{f}[Bc][Bc]{\scalebox{.9}{$f$}}
\psfrag{Y}[Bl][Bl]{\scalebox{.8}{$Y$}}
\psfrag{U}[Bl][Bl]{\scalebox{.8}{$U$}}
\psfrag{g}[Bc][Bc]{\scalebox{.9}{$g$}}
\psfrag{V}[Bl][Bl]{\scalebox{.8}{$V$}}
\rsdraw{.45}{.9}{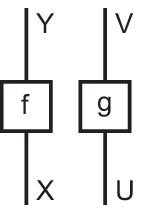} \, = \, \rsdraw{.45}{.9}{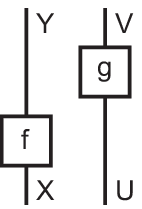} \, = \, \rsdraw{.45}{.9}{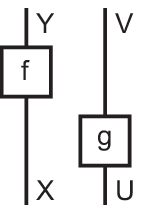}
$$
which graphically expresses the formulas
$$
f\otimes g =(\id_Y\otimes g) (f \otimes \id_U)= (f\otimes \id_V)( \id_X\otimes g).
$$
Here and in the sequel, the equality sign between  the diagrams     means the equality of the   corresponding morphisms.

The braiding $\tau$ of a braided category, and its inverse, are depicted as
$$
\psfrag{Z}[Bl][Bl]{\scalebox{.8}{$X$}}
\psfrag{X}[Br][Br]{\scalebox{.8}{$X$}}
\psfrag{A}[Bl][Bl]{\scalebox{.8}{$Y$}}
\psfrag{Y}[Br][Br]{\scalebox{.8}{$Y$}}
\tau_{X,Y}=\,\; \rsdraw{.45}{.9}{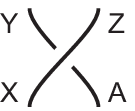}
\quad \text{and} \quad
\tau^{-1}_{Y,X}=\,\; \rsdraw{.45}{.9}{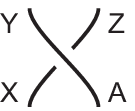} \,.
$$
The axioms of a braiding are depicted as follows: for all objects $X,Y,Z$,
\begin{center}
\psfrag{Z}[Bl][Bl]{\scalebox{.8}{$X \otimes Y$}}
\psfrag{X}[Br][Br]{\scalebox{.8}{$X\otimes Y$}}
\psfrag{A}[Bl][Bl]{\scalebox{.8}{$Z$}}
\psfrag{Y}[Br][Br]{\scalebox{.8}{$Z$}}
\rsdraw{.45}{.9}{braid} \qquad = \;
\psfrag{Z}[Bl][Bl]{\scalebox{.8}{$X$}}
\psfrag{C}[Bl][Bl]{\scalebox{.8}{$Y$}}
\psfrag{X}[Br][Br]{\scalebox{.8}{$X$}}
\psfrag{B}[Br][Br]{\scalebox{.8}{$Y$}}
\psfrag{A}[Bl][Bl]{\scalebox{.8}{$Z$}}
\psfrag{Y}[Br][Br]{\scalebox{.8}{$Z$}}
\rsdraw{.45}{.9}{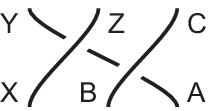} \quad \text{and} \quad \quad
\psfrag{Z}[Bl][Bl]{\scalebox{.8}{$X$}}
\psfrag{X}[Br][Br]{\scalebox{.8}{$X$}}
\psfrag{A}[Bl][Bl]{\scalebox{.8}{$Y \otimes Z$}}
\psfrag{Y}[Br][Br]{\scalebox{.8}{$Y \otimes Z$}}
\rsdraw{.45}{.9}{braid} \qquad = \;
\psfrag{Z}[Bl][Bl]{\scalebox{.8}{$Y$}}
\psfrag{C}[Bl][Bl]{\scalebox{.8}{$Z$}}
\psfrag{X}[Br][Br]{\scalebox{.8}{$Y$}}
\psfrag{B}[Br][Br]{\scalebox{.8}{$Z$}}
\psfrag{A}[Bl][Bl]{\scalebox{.8}{$X$}}
\psfrag{Y}[Br][Br]{\scalebox{.8}{$X$}}
\rsdraw{.45}{.9}{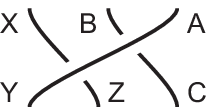} \;.
\end{center}
The naturality of $\tau$ ensures that we can push boxes across a strand without changing the morphism represented by the diagram: for any morphism $f$,
$$
\psfrag{f}[Bc][Bc]{\scalebox{.9}{$f$}}
\rsdraw{.45}{.9}{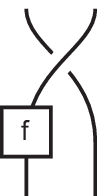} \; = \; \rsdraw{.45}{.9}{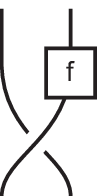} \quad \text{and} \quad
\rsdraw{.45}{.9}{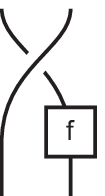} \; = \; \rsdraw{.45}{.9}{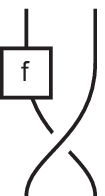} \;.
$$

The  above  graphical calculus may be  enhanced for a pivotal category by orienting all arcs in the diagrams and depicting the (co)evaluations as
$$
\psfrag{X}[Bc][Bc]{\scalebox{.7}{$X$}} \lev_X= \rsdraw{.45}{.9}{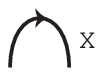}\, \,,\quad
\rev_X= \rsdraw{.45}{.9}{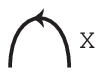}   \,, \quad
\lcoev_X= \rsdraw{.45}{.9}{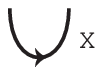} \,,\quad
\psfrag{C}[Bc][Bc]{\scalebox{.7}{$X$}} \rcoev_X=
\rsdraw{.45}{.9}{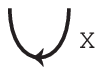}\,.
$$
Here, an arc  colored with an object $X$ and oriented   downward (resp., upward) contributes~$X$   (resp., $X^*$)  to the source/target of  morphisms.
For example, a morphism $f \colon X^* \otimes Y \to A \otimes B^* \otimes C$  may  be represented as
$$
\psfrag{X}[Bl][Bl]{\scalebox{.8}{$X$}}
\psfrag{h}[Bc][Bc]{\scalebox{.9}{$f$}}
\psfrag{Y}[Bl][Bl]{\scalebox{.8}{$Y$}}
\psfrag{A}[Bl][Bl]{\scalebox{.8}{$A$}}
\psfrag{B}[Bl][Bl]{\scalebox{.8}{$B$}}
\psfrag{C}[Bl][Bl]{\scalebox{.8}{$C$}}
\rsdraw{.45}{.9}{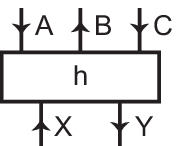}\,.
$$
In particular, the identity of the dual $X^*$ of any object $X$ is represented as  %\CUTav
$$
\id_{X^*}= \,
\psfrag{X}[Bl][Bl]{\scalebox{.8}{$X^*$}}
\rsdraw{.45}{.9}{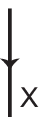}  =  \; \psfrag{X}[Bl][Bl]{\scalebox{.8}{$X$}}
\rsdraw{.45}{.9}{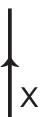} \;.
$$
The duality identities are    graphically expressed as
$$
 \psfrag{X}[Bc][Bc]{\scalebox{.8}{$X$}}
 \psfrag{u}[Bc][Bc]{\scalebox{1.111111}{$=$}}
 \rsdraw{.45}{.9}{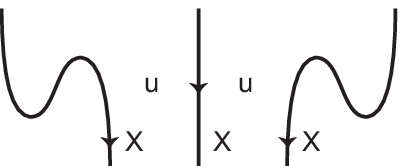} \qquad \text{and} \quad  \rsdraw{.45}{.9}{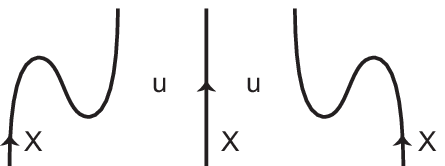} \;.
$$
In a pivotal category, the morphism represented by a diagram is preserved under ambient isotopies of the diagram keeping
fixed the bottom and  top endpoints.

\section{Categorical Hopf algebras}\label{sect-3}
In this section, we review the  categorical version of the notions of a Hopf algebra and its modules.

\subsection{Categorical algebras}\label{sect-categorical-algebras}
Let $\cc$ be a monoidal category. An \emph{algebra in~$\cc$} is an object $A$ of $\cc$ endowed with morphisms $m\co A \otimes A \to A$ (the  product) and $u\co \un \to A$ (the unit) such that
\begin{equation*}
m(m \otimes \id_A)=m(\id_A \otimes m) \quad \text{and} \quad m(\id_A \otimes u)=\id_A=m(u \otimes \id_A).
\end{equation*}
We depict the product and  the unit  as
\begin{center}
\psfrag{A}[Bc][Bc]{\scalebox{.8}{$A$}} $m=$\rsdraw{.45}{.9}{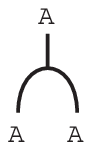} \quad \text{and} \quad $u=$\,\rsdraw{.35}{.9}{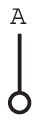}\;.
\end{center}
The   axioms above have the following   graphical interpretation:
\begin{center}
\psfrag{A}[Bc][Bc]{\scalebox{.8}{$A$}} $\rsdraw{.45}{.9}{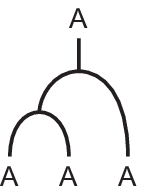}=\rsdraw{.45}{.9}{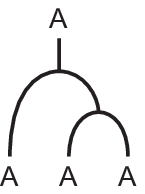}$ \quad \text{and} \quad $\rsdraw{.45}{.9}{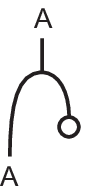}=\rsdraw{.45}{.9}{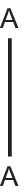}=\rsdraw{.45}{.9}{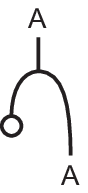}$\;.
\end{center}
Note that algebras in $\cc$ are also called \emph{monoids} in $\cc$ in the literature.

\subsection{Categorical coalgebras}\label{sect-categorical-coalgebras}
A \emph{coalgebra} in a monoidal category $\cc$ is an algebra in   the opposite category $\cc^\opp=(\cc^\opp,\otimes,\un)$.  In other words,  a coalgebra in~$\cc$ is   an object~$A$ of~$\cc$ endowed with morphisms $\Delta\co A \to A \otimes A$ (the coproduct) and
 $\varepsilon\co A \to \un$ (the counit) such that
\begin{equation*}
(\Delta \otimes \id_A)\Delta=(\id_A \otimes \Delta)\Delta \quad \text{and} \quad (\id_A \otimes \varepsilon)\Delta=\id_A=(\varepsilon \otimes \id_A)\Delta.
\end{equation*}
We depict the coproduct and the counit as
\begin{center}
\psfrag{A}[Bc][Bc]{\scalebox{.8}{$A$}} $\Delta=$\rsdraw{.45}{.9}{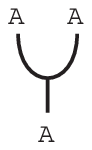} \quad \text{and} \quad $\varepsilon=$\,\rsdraw{.45}{.9}{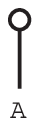}\;.
\end{center}
The axioms above  are  depicted as
\begin{center}
\psfrag{A}[Bc][Bc]{\scalebox{.8}{$A$}} $\rsdraw{.45}{.9}{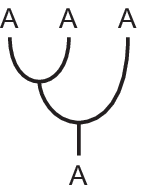}=\rsdraw{.45}{.9}{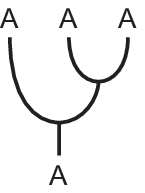}$ \quad \text{and} \quad $\rsdraw{.45}{.9}{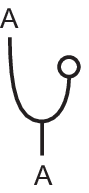}=\rsdraw{.45}{.9}{uA-axiom3}=\rsdraw{.45}{.9}{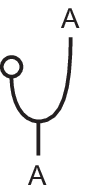}$\;.
\end{center}

For example, in a cartesian monoidal category (see Section~\ref{sect-cartesian-monoidal-categories}), any object is a coalgebra with coproduct being the diagonal map and counit  being the augmentation. In fact, any coalgebra in a cartesian monoidal category is of this form.

\subsection{Categorical bialgebras}\label{sect-categorical-bialgebras}
To define bialgebras in a monoidal category, we need compatibility conditions between product and coproduct, and the formulation of one of the  conditions requires a substitute for the flip map which can be provided by a braiding. A \emph{bialgebra}  in a braided category~$\bb$ is an object $A$ of~$\bb$ endowed with an algebra structure $(m,u)$ and a coalgebra structure $(\Delta,\varepsilon)$ in~$\bb$ satisfying the following conditions (expressing that $\Delta$ and $\varepsilon$ are algebra morphisms or, equivalently, that $m$ and $u$ are coalgebra morphisms):
\begin{align*}
\Delta m&=(m \otimes m)(\id_A \otimes \tau_{A,A} \otimes \id_A)(\Delta \otimes \Delta),  & \Delta u&=u \otimes u, \\  \varepsilon m&=\varepsilon \otimes \varepsilon, & \varepsilon u&=\id_\un,
\end{align*}
where $\tau$ is the braiding of $\bb$. Pictorially,
\begin{center}
\psfrag{A}[Bc][Bc]{\scalebox{.8}{$A$}} \psfrag{B}[Br][Br]{\scalebox{.8}{$A$}}
$\rsdraw{.45}{.9}{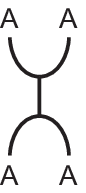} = \;\rsdraw{.45}{.9}{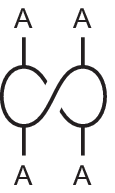} \;, \qquad \rsdraw{.39}{.9}{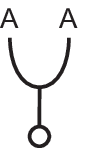} = \; \rsdraw{.39}{.9}{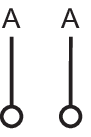} \,, \qquad \rsdraw{.45}{.9}{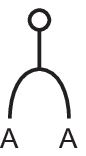} = \,\rsdraw{.45}{.9}{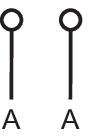}\,, \qquad\rsdraw{.4}{.9}{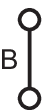} = \emptyset$.%\id_\un $.
\end{center}

\subsection{Categorical Hopf algebras}\label{sect-categorical-Hopf-algebras}
Let $\bb$ be a braided category. An \emph{antipode} for a bialgebra $A=(A, m,u, \Delta,\varepsilon)$ in $\bb$ is a morphism $S\co A \to A$ such that
\begin{equation*}
m(S \otimes \id_A)\Delta=u \varepsilon=m(\id_A \otimes S)\Delta.
\end{equation*}
This axiom is  depicted as
\begin{center}
\psfrag{A}[Bc][Bc]{\scalebox{.8}{$A$}}
$\rsdraw{.45}{.9}{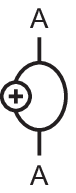}\,=\,\rsdraw{.45}{.9}{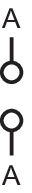}\,=\,\rsdraw{.45}{.9}{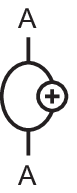}$ \quad \text{where} \;\;
$S=\rsdraw{.45}{.9}{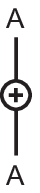}$\;.
\end{center}
If it exists, an antipode is unique and is anti-multiplicative and anti-comultiplicative:
\begin{equation*}
Sm=m(S \otimes S)\tau_{A,A},\quad Su=u, \quad \Delta S=\tau_{A,A} (S \otimes S)\Delta, \quad \varepsilon S=\varepsilon.
\end{equation*}
Pictorially,
\begin{center}
\psfrag{A}[Bc][Bc]{\scalebox{.8}{$A$}}
$\rsdraw{.45}{.9}{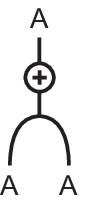}=\;\,\rsdraw{.45}{.9}{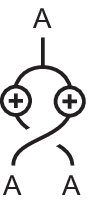}$  \;, \qquad
$\rsdraw{.4}{.9}{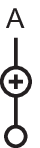}\,=\,\rsdraw{.4}{.9}{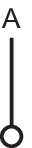}$\;, \qquad
$\rsdraw{.45}{.9}{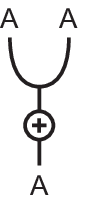}=\;\,\rsdraw{.45}{.9}{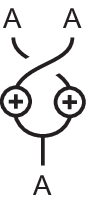}$ \;, \qquad
$\rsdraw{.5}{.9}{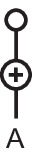}\,=\,\rsdraw{.5}{.9}{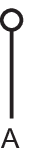}$\;.
\end{center}
When the antipode $S$ is invertible,
we depict  its inverse   $S^{-1} \co A \to A$    as
\begin{center}
\psfrag{A}[Bc][Bc]{\scalebox{.8}{$A$}}
$S^{-1}= \rsdraw{.45}{.9}{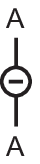}$\,, \quad so that \; \, $\rsdraw{.45}{.9}{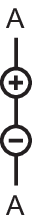}=\,\rsdraw{.45}{.9}{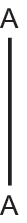}=\,\rsdraw{.45}{.9}{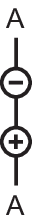}$\;.
\end{center}

A \emph{Hopf algebra in $\bb$} is a bialgebra in $\bb$ which admits an invertible antipode.
Note that the notion of a Hopf algebra is self-dual:  $(A,m,u,\Delta,\varepsilon,S)$ is a Hopf algebra in~$\bb$ if and only if $(A,\Delta,\varepsilon,m,u,S)$ is a Hopf algebra in the opposite category~$\bb^\opp$.

A \emph{grouplike element} of Hopf algebra $A$ in $\bb$ is a morphism $G \co \un \to A$ such that
$$
\Delta G = G \otimes G \quad \text{and} \quad \varepsilon G=\id_\un.
$$
Such a $G$ is invertible in the monoid $(\Hom_\bb(\un,A),\ast,u)$, where $\alpha \ast \beta=m(\alpha \otimes \beta)$, and its inverse, denoted by $G^{-1}$, is also a grouplike element of $A$ and is computed by $G^{-1}=SG=S^{-1}G$. In particular, the set of grouplike elements of $A$ is a group (with product $\ast$ and unit $u$).

\subsection{Examples}\label{sect-exemples-Hopf-algebras}
1. Hopf algebras in the symmetric monoidal category of  \kt modules and \kt linear homomorphisms are  the usual Hopf \kt  algebras.

2. Hopf algebras in the symmetric monoidal category of super \kt modules and grading-preserving \kt linear homomorphisms are  the usual super Hopf \kt  algebras.

3. Any group object  in a cartesian monoidal category becomes a Hopf algebra, with its canonical coalgebra structure (see Section~\ref{sect-categorical-coalgebras}) and with antipode given by the group inversion. This induces a bijective correspondence between group objects and Hopf algebras in a cartesian monoidal category. For example, Hopf algebras in $\Set$ are groups. More generally, Hopf algebras in $\Cat$ are crossed modules, as detailed in Section~\ref{sect-crossed-modules-as-Hopf-algebras}.

\subsection{Modules in categories}\label{sect-modules-in-cat}
Let $(A,m,u)$ be an algebra in a monoidal category~$\cc$.
A \emph{left $A$-module} (in $\cc$) is a pair~$(M,r)$, where $M$ is an object of $\cc$ and $r\co A \otimes M \to M$ is a morphism in $\cc$, such that
\begin{equation*}
r(m \otimes \id_M)=r(\id_A \otimes r) \quad \text{and} \quad r(u \otimes \id_M)=\id_M.
\end{equation*}
Graphically, these conditions  are depicted as
\begin{center}
\psfrag{A}[Bc][Bc]{\scalebox{.8}{$A$}}\psfrag{N}[Bc][Bc]{\scalebox{.8}{$M$}}
$\rsdraw{.45}{.9}{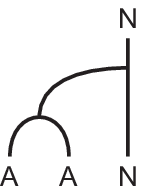}=\rsdraw{.45}{.9}{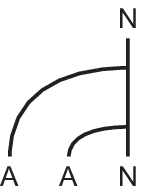}$ \quad \text{and} \quad $\rsdraw{.45}{.9}{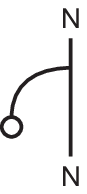}=\,\rsdraw{.45}{.9}{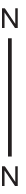}$\;\;, \;\; \text{where} \;\;
$r=\rsdraw{.45}{.9}{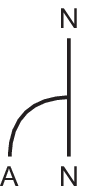}$\;.
\end{center}
One can similarly introduce right $A$-modules,  but we will not use them. From now on, by an $A$-module, we mean a left $A$-module.

An \emph{$A$-linear morphism} between two $A$-modules $(M,r)$ and $(N,s)$  is a morphism $f\co M \to N$ in $\cc$ such that $fr=s(\id_A \otimes f)$, that is, pictorially,
\begin{center}
\psfrag{A}[Bc][Bc]{\scalebox{.8}{$A$}}\psfrag{B}[Bc][Bc]{\scalebox{.8}{$M$}}\psfrag{N}[Bc][Bc]{\scalebox{.8}{$N$}} \psfrag{f}[Bc][Bc]{\scalebox{.9}{$f$}}
$\rsdraw{.45}{.9}{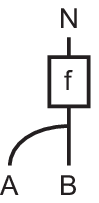}\;=\rsdraw{.45}{.9}{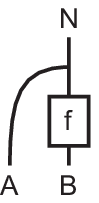}$\;.
\end{center}

We let  $\rmod{\cc}{(A)}$ be the category of $A$-modules and $A$-linear morphisms, with composition inherited from $\cc$.  The \emph{forgetful functor} $\rmod{\cc}{(A)} \to \cc$ carries  any $A$-module $(M,r)$ to $M$ and   any $A$-linear morphism to itself.

If  $A=(A, m,u, \Delta,\varepsilon)$ is  a bialgebra in a braided category $\bb$, then the category $\rmod{\bb}{(A)}$ of $A$-modules    has a canonical structure of a  monoidal  category. Its unit object is the pair   $(\un,\varepsilon)$.  Its monoidal product is given on the objects by
$$
(M,r) \otimes (N,s)=(M \otimes N,t)
$$
where
$$
t= (r \otimes s)(\id_A \otimes \tau_{A,M} \otimes \id_N)(\Delta \otimes \id_{M \otimes N}) =
\psfrag{X}[Bc][Bc]{\scalebox{.8}{$M$}} \psfrag{Y}[Bc][Bc]{\scalebox{.8}{$N$}} \psfrag{A}[Bc][Bc]{\scalebox{.8}{$A$}}
\psfrag{n}[Bc][Bc]{\scalebox{.9}{$r$}}\psfrag{s}[Bc][Bc]{\scalebox{.9}{$s$}}
\rsdraw{.45}{.9}{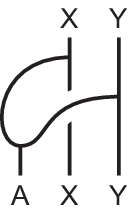}
$$
and on the morphisms by the monoidal product in $\cc$. Note that the forgetful functor $ \rmod{\bb}{(A)} \to \bb$ is   strict monoidal.

Assume that $\bb$ is a closed braided category. Then a bialgebra $A$ in $\bb$ is a Hopf algebra if and only if the monoidal category $\rmod{\bb}{(A)}$ is closed and the forgetful functor $\rmod{\bb}{(A)} \to \bb$ preserves the internal Homs. (This follows from Theorem~3.6 and Remark 5.6 of \cite{BLV}.) If such is the case, then the left and right internal Homs between $A$-modules $(M,r)$ and $(N,s)$ are
\begin{align*}
& \lhom{(M,r),(N,s)}=\left(\lhom{M,N}, \lhom{\id_M,\Theta_l}\lcoeva{M}{A\otimes \lhom{M,N}} \right), \\
& \rhom{(M,r),(N,s)}=\left(\rhom{M,N}, \rhom{\id_M,\Theta_r} \rcoeva{M}{A\otimes \rhom{M,N}}\right),
\end{align*}
with (co)evaluations inherited from $\bb$, where
$$
\psfrag{a}[cc][cc]{\scalebox{.9}{$\leva{M}{N}$}}
\psfrag{X}[Bc][Bc]{\scalebox{.8}{$\lhom{M,N}$}}
\psfrag{Y}[Bc][Bc]{\scalebox{.8}{$M$}}
\psfrag{Z}[Bc][Bc]{\scalebox{.8}{$N$}}
\psfrag{A}[Bc][Bc]{\scalebox{.8}{$A$}}
\Theta_l=\;\rsdraw{.45}{.9}{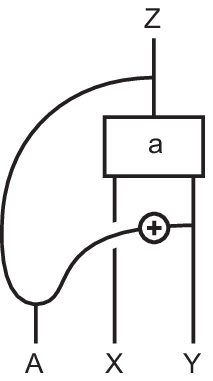}
\qquad \text{and} \qquad
\psfrag{a}[cc][cc]{\scalebox{.9}{$\reva{M}{N}$}}
\psfrag{X}[Bc][Bc]{\scalebox{.8}{$\rhom{M,N}$}}
\psfrag{Y}[Bc][Bc]{\scalebox{.8}{$M$}}
\psfrag{Z}[Bc][Bc]{\scalebox{.8}{$N$}}
\psfrag{A}[Bc][Bc]{\scalebox{.8}{$A$}}
\Theta_r=\rsdraw{.45}{.9}{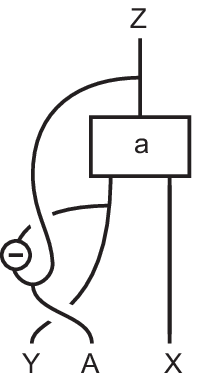}\;.
$$

Assume that $\bb$ is a rigid braided category. Then a bialgebra $A$ in $\bb$ is a Hopf algebra if and only if the monoidal category $\rmod{\bb}{(A)}$ is rigid. (This follows from Theorem 3.8 and Example 3.10 of \cite{BV}.)
If such is the case, then the left and right duals of an $A$-module $(M,r)$ are
$$
\ldual{(M,r)}=(\ldual{M},\leftidx{^\circ}{r}{}) \quad \text{and} \quad \rdual{(M,r)}=(\rdual{M},r^\circ)
$$
where
$$
\leftidx{^\circ}{r}{}=\!\!
\psfrag{a}[cc][cc]{\scalebox{.9}{$\lev_{M}$}}
\psfrag{u}[cc][cc]{\scalebox{.9}{$\lcoev_{M}$}}
\psfrag{X}[Bc][Bc]{\scalebox{.8}{$\ldual{M}$}}
\psfrag{A}[Bc][Bc]{\scalebox{.8}{$A$}}
\psfrag{n}[Bc][Bc]{\scalebox{.9}{$r$}}
\rsdraw{.45}{.9}{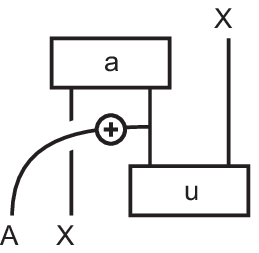}
\quad \;\text{and} \;\quad r^\circ=\!\!
\psfrag{a}[Bc][Bc]{\scalebox{.9}{$\rev_{M}$}}
\psfrag{u}[Bc][Bc]{\scalebox{.9}{$\rcoev_{M}$}}
\psfrag{X}[Bc][Bc]{\scalebox{.8}{$\rdual{M}$}}
\psfrag{A}[Bc][Bc]{\scalebox{.8}{$A$}}
\psfrag{n}[Bc][Bc]{\scalebox{.9}{$r$}}
\rsdraw{.45}{.9}{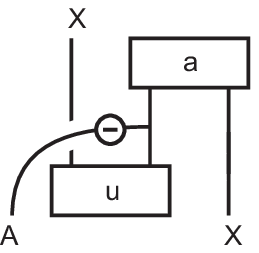}\; ,
$$
with (co)evaluations inherited from $\bb$:
\begin{align*}
\lev_{(M,r)} & =\lev_M, && \lcoev_{(M,r)}=\lcoev_M, \\
\rev_{(M,r)} & =\rev_M, && \rcoev_{(M,r)}=\rcoev_M.
\end{align*}

Assume that $\bb$ is a pivotal braided category. The (right) \emph{twist} of $\bb$ is the natural automorphism  $\theta=\{\theta_X \co X \to X\}_{X \in \bb}$ defined by
$$
\theta_X  =\!
\psfrag{X}[Br][Br]{\scalebox{.8}{$X$}}\rsdraw{.42}{.9}{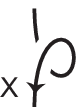}\;=(\id_X
\otimes \rev_X)(\tau_{X,X} \otimes \id_{X^*})(\id_X \otimes
\lcoev_X).
$$
Let $A$ be a Hopf algebra in $\bb$. Then the pivotal structures on the monoidal category  $\rmod{\bb}{(A)}$ are in bijection with the pairs $(G,\gamma)$, where~$G$ is a grouplike element of $A$ (see Section~\ref{sect-categorical-Hopf-algebras}) and $\gamma=\{\gamma_X \co X \to X\}_{X \in \bb}$ is a monoidal natural automorphism, % (of the identity functor $\bb \to \bb$), satisfying
such that the square of the antipode $S$ of $A$ satisfies
$$
S^2= \theta_A \circ \mathrm{Ad}_G  \circ \gamma_A \quad \text{where} \quad
\psfrag{A}[Bl][Bl]{\scalebox{.8}{$A$}}
\psfrag{G}[Bc][Bc]{\scalebox{.9}{$G$}}
\psfrag{H}[Bc][Bc]{\scalebox{.9}{$G^{-1}$}}
\psfrag{a}[Bc][Bc]{\scalebox{.9}{$\theta_A$}}
\psfrag{x}[Bc][Bc]{\scalebox{.9}{$\gamma_A$}}
\mathrm{Ad}_G= \,\rsdraw{.45}{.9}{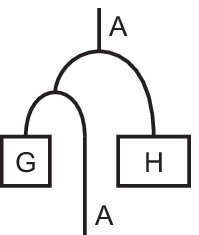} \;\,.
$$
(This follows from Proposition 7.6 and Example 7.2 of \cite{BV}.)
The pivotal structure on $\rmod{\bb}{(A)}$ associated with such a pair $(G,\gamma)$ is given for any $A$-module $(M,r)$~by
$$
(M,r)^*=(M^*,r^\dagger) \quad \text{where} \quad
\psfrag{X}[Bl][Bl]{\scalebox{.8}{$M$}}
\psfrag{Y}[Br][Br]{\scalebox{.8}{$M$}}
\psfrag{A}[Br][Br]{\scalebox{.8}{$A$}}
\psfrag{n}[Bc][Bc]{\scalebox{.9}{$r$}}
r^\dagger=\!\!\!\rsdraw{.45}{.9}{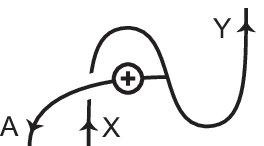}\;, %=\;\rsdraw{.45}{.9}{A-mod-right-dual-pivotal-2}\; ,
$$
with left (co)evaluations given by $\lev_{(M,r)} =\lev_M$ and $\lcoev_{(M,r)}=\lcoev_M$, and right (co)evaluations given by
$$
\psfrag{X}[Bl][Bl]{\scalebox{.8}{$M$}}
\psfrag{A}[Br][Br]{\scalebox{.8}{$A$}}
\psfrag{G}[Bc][Bc]{\scalebox{.9}{$G$}}
\psfrag{H}[Bc][Bc]{\scalebox{.9}{$G^{-1}$}}
\psfrag{u}[Bc][Bc]{\scalebox{.9}{$\gamma_M$}}
\psfrag{v}[Bc][Bc]{\scalebox{.9}{$\gamma_M^{-1}$}}
\rev_{(M,r)}=\;\rsdraw{.45}{.9}{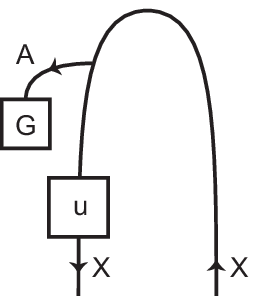} \quad \text{and} \quad\; \rcoev_{(M,r)}=\,\rsdraw{.45}{.9}{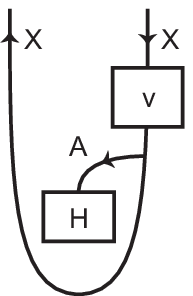} \;.
$$
Note that the forgetful functor $\rmod{\bb}{(A)} \to \bb$ is then pivotal if and only if for all $A$-module $(M,r)$,
$$
\psfrag{X}[Bl][Bl]{\scalebox{.8}{$M$}}
\psfrag{A}[Br][Br]{\scalebox{.8}{$A$}}
\psfrag{G}[Bc][Bc]{\scalebox{.9}{$G$}}
\psfrag{u}[Bc][Bc]{\scalebox{.9}{$\gamma_M$}}
\rsdraw{.45}{.9}{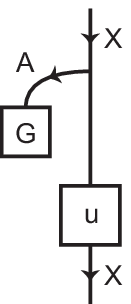} \; = \;\; \rsdraw{.45}{.9}{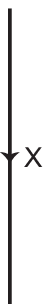} \;\;.
$$
In particular, if $A$ is \emph{involutory} in the sense that  its antipode~$S$ satisfies $S^2=\theta_A$, then the monoidal category  $\rmod{\bb}{(A)}$ carries a structure of a pivotal category so that the forgetful functor $\rmod{\bb}{(A)} \to \bb$ is pivotal (by taking $G=u$ and $\gamma=\id_\bb$).

\subsection{Comodules in categories}\label{sect-comodules-in-cat}
Given a coalgebra $A$ in a monoidal category~$\cc$ (which is an algebra in the opposite category $\cc^\opp$), we define the category of left $A$-comodules in $\cc$ by setting $$\mathrm{Comod}_\cc(A)=\bigl(\mathrm{Mod}_{\cc^\opp}(A)\bigr)^\opp.$$
In particular its objects are pairs~$(M,\delta)$, with $M$ an object of $\cc$ and $\delta\co M \to A \otimes M$ a morphism in $\cc$ such that
\begin{equation*}
(\Delta \otimes \id_M)\delta=(\id_A \otimes \delta)\delta \quad \text{and} \quad (\varepsilon \otimes \id_M)\delta=\id_M.
\end{equation*}
We deduce from Section~\ref{sect-modules-in-cat} that the category $\mathrm{Comod}_\cc(A)$ is monoidal when $\cc$ is braided and $A$ is a bialgebra, and that $\mathrm{Comod}_\cc(A)$ is rigid (resp.\@ closed) when $\cc$ is braided rigid (resp.\@ braided closed) and $A$ is a Hopf algebra. Also,  if $\cc$ is pivotal braided and $A$ is a Hopf algebra, then the pivotal structures on the monoidal category  $\mathrm{Comod}_\cc(A)$ are in bijection with the pairs $(\alpha,\gamma)$, where~$\alpha \co A \to \un$ is an algebra morphism and $\gamma=\{\gamma_X \co X \to X\}_{X \in \bb}$ is a monoidal natural automorphism, % (of the identity functor $\bb \to \bb$), satisfying
such that $S^2= \theta_A \circ \mathrm{Ad}_\alpha \circ \gamma_A$, where $\theta$ is the twist of $\cc$ and
$$
\psfrag{A}[Bl][Bl]{\scalebox{.8}{$A$}}
\psfrag{G}[Bc][Bc]{\scalebox{.9}{$\alpha$}}
\mathrm{Ad}_\alpha= \,\rsdraw{.45}{.9}{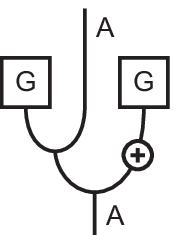} \;\,.
$$

\section{Crossed modules}\label{sect-4}
In this section, we quickly review crossed modules and their relationship with Hopf algebras in the category of small categories.

\subsection{Crossed modules}\label{sect-crossed-modules-def}
A \emph{crossed module} is a group homomorphism $\CM \co E \to H$ together with a left  action of $H$ on $E$ (by group automorphisms) denoted  $$(x,e) \in H \times E  \mapsto \elact{x}{e} \in E $$ such that $\CM$ is equivariant with respect to the conjugation action of $H$ on itself and satisfies the Peiffer identity, that is, for all $x \in H$ and $e,f \in E$,
$$
\CM(\elact{x}{e})=x\CM(e)x^{-1} \quad \text{and} \quad \lact{{\CM(e)}}{\!f}=efe^{-1}.
$$
These axioms imply that the image $\Ima(\CM)$ is normal in $H$ and that the kernel $\Ker(\CM)$ is central in $E$ and is acted on trivially by $\Ima(\CM)$. In particular,  $\Ker(\CM)$ inherits an action of $H/\Ima(\CM)=\Coker(\CM)$.

A \emph{morphism} from a crossed module $\CM \co E \to H$ to a crossed module $\CM' \co E' \to H'$ is a pair $(\psi \co E \to E',\varphi \co H \to H')$
of group homomorphisms such that
$$
\CM'\bigl(\psi(e) \bigr)=\varphi \bigl( \CM(e) \bigr) \quad \text{and} \quad \psi(\elact{x}{e})=\lact{\varphi(x)}{\psi(e)}
$$
for all $e\in E$ and $x \in H$.

\subsection{Examples}\label{sect-crossed-modules-ex}
1. Given any  normal subgroup $E$ of a group $H$, the inclusion $E\hookrightarrow H$ is a crossed module with the conjugation action of $H$ on $E$.

2. For any group $E$, the homomorphism $E \to \Aut(E)$ sending any element of $E$ to the corresponding inner automorphism is a crossed module.

3. If $E$ is an abelian group, then the trivial map $E \to 1$ is a crossed module.

4. A key geometric example of a crossed module is due to Whitehead: if $X$ is a topological space, $Y$ is a subspace of $X$, and $y$ is a point of $Y$, then the homotopy boundary map $\partial \co \pi _{2}(X,Y,y)\rightarrow \pi _{1}(Y,y)$,
together with the standard action of $\pi _{1}(Y,y)$ on $\pi _{2}(X,Y,y)$,  is a crossed module.

\subsection{Crossed modules as Hopf algebras}\label{sect-crossed-modules-as-Hopf-algebras}
By \cite{BS}, crossed modules are group objects in the category $\Cat$ of small categories and functors (endowed with its canonical cartesian monoidal structure). Also, by the third example in Section \ref{sect-exemples-Hopf-algebras}, there is a bijective correspondence between group objects and Hopf algebras in a cartesian monoidal category.  Consequently, Hopf algebras in $\Cat$ are in bijective correspondence with crossed modules.

Explicitly, the Hopf algebra
$\tg_\CM=(\tg_\CM,m_\CM,u_\CM,\Delta_\CM,\varepsilon_\CM,S_\CM)$ in $\Cat$ associated with a crossed module $\CM \co E \to H$  is described as follows.  The objects of~$\tg_\CM$  are the elements of~$H$. For any objects $x,y \in H$,
$$
\Hom_{\tg_\CM}(x,y)=\{e \in E \, | \, y=\CM(e)x\}.
$$
The composition of morphisms is given by the product of $E$:
$$
\left(y \xrightarrow{f} z \right) \circ \left(x \xrightarrow{e} y \right) = \left(x \xrightarrow{fe} z \right) \quad \text{and} \quad \id_x=1 \in E.
$$
Note that $\tg_\CM$ is a groupoid. The product $m_\CM \co \tg_\CM \times \tg_\CM \to \tg_\CM$ of $\tg_\CM$ is defined on objects  and morphisms by
$$
m_\CM(x,y)=xy \quad \text{and} \quad
m_\CM \!\left(x \xrightarrow{e} y ,z \xrightarrow{f} t \right) = \left(xz \xrightarrow{e\lact{x}{\!f}} yt \right).
$$
Denote by $\mathbf{1}$ the trivial category with a single object $\ast$ and a single morphism $\id_\ast$. The unit $u_\CM \co \mathbf{1} \to \tg_\CM$
is defined by
$$
u_\CM (\ast)=1 \in H.
$$
The coproduct $\Delta_\CM \co \tg_\CM \to \tg_\CM \times \tg_\CM $ and counit $\varepsilon_\CM \co \tg_\CM \to \mathbf{1}$ are the diagonal and augmentation: for any object $x$ and morphism $e$,
$$
\Delta_\CM (x)=(x,x), \quad \Delta_\CM (e)=(e,e), \quad \varepsilon_\CM (x)=\ast, \quad \varepsilon_\CM (e)=\id_\ast.
$$
The antipode $S_\CM \co \tg_\CM \to\tg_\CM$ is involutive ($S_\CM^{-1}=S_\CM$) and is computed by
$$
S_\CM(x)=x^{-1} \quad \text{and} \quad
S_\CM \!\left(x \xrightarrow{e} y \right) = \left(x^{-1} \xrightarrow{\lact{x^{-1}}{\!(e^{-1})}} y^{-1} \right).
$$

\section{Graded monoidal categories}\label{sect-5}

In this section, we review the notions of a monoidal category graded by a group or a crossed module. We refer to \cite{SV} for details.

\subsection{Linear categories}\label{sect-linear-categories}
A category~$\cc$ is \emph{\kt linear}  if for all objects $X,Y\in \cc$, the set $\Hom_\cc(X,Y)$ carries a structure of a left  \kt module so that  the composition of morphisms  is $\kk$-bilinear. An object $X$ of a \kt linear category $\cc$ is called a \emph{zero object}  if $\id_X=0$.
A zero object, if it exists, is unique up to  isomorphism.

A monoidal category is  \emph{$\kk$-linear} if it is $\kk$-linear as a category and the monoidal product of morphisms
is \kt bilinear.

\subsection{Monoidal categories graded by a group}\label{sect-monoidal-graded-groups}
Let $H$ be a group. An \emph{$H$-graded monoidal category (over $\kk$)} is a $\kk$-linear  mo\-noi\-dal category $\cc=(\cc,\otimes,\un)$ endowed with a family $\{\cc_h\}_{h \in H}$ of full subcategories such that:
\begin{enumerate}
  \labela
  \item  Each object of $\cc$ is a direct sum of objects in $\bigcup_{h \in H}\cc_h$.
  \item  For all $X\in \cc_h$ and $Y\in \cc_k$ with $h\neq k$, $$\Hom_{\cc} (X,Y)=0.$$
  \item  For all $X\in \cc_h$ and $Y\in \cc_k$, we have:  $X \otimes Y \in \cc_{hk}$.
  \item  $\un \in \cc_1$.
\end{enumerate}
In this case, we write by abuse\footnote{This is a genuine direct sum when $\cc$ is additive and each $\cc_h$ contains a zero object.} of notation:
$$
\cc=\bigoplus_{h \in H} \cc_h.
$$
The monoidal subcategory $\cc_1$ is called the \emph{neutral component} of $\cc$.
An object $X$ of~$\cc$ is   \emph{homogeneous} if it is nonzero and  $X\in \cc_h$ for some $h\in H$. Such an $h$  is then uniquely determined by $X$, denoted by  $|X|$, and called the \emph{degree} of $X$.

\subsection{Hom-graded categories}\label{sect-Hom-graded-categories}
Let $E$ be a group with unit $1$. An \emph{$E$-Hom-graded category (over $\kk$)} is a category enriched over the monoidal category of
$E$-graded \kt modules  and  \kt linear  grading-preser\-ving homomorphisms. Explicitly, this is a 
\kt linear category $\cc$ such that:
\begin{enumerate}
\labela
\item The Hom-sets in $\cc$ are $E$-graded \kt modules: for all objects $X,Y \in \cc$,
$$
\Hom_\cc(X,Y)=\bigoplus_{e \in E} \Hom_\cc^e(X,Y).
$$
\item The composition in $\cc$ is multiplicative with respect to the degree: for all $e,f \in E$ and  $X,Y, Z \in \cc$, it sends $\Hom_\cc^f(Y,Z) \times \Hom_\cc^e(X,Y)$ into $\Hom_\cc^{fe}(X,Z)$.
\item The identities have trivial degree: for all $X\in \cc$, $$\id_X \in \End_\cc^1(X)=\Hom_\cc^1(X,X).$$
\end{enumerate}

Let $\cc$ be an $E$-Hom-graded category.
A morphism $\alpha \co X \to Y$ in $\cc$ is \emph{homogeneous of degree $e \in E$} if $\alpha \in \Hom_\cc^e(X,Y)$. Note that if $\alpha$ is nonzero, then such an $e \in E$ is unique, is called the \emph{degree} of $\alpha$, and is denoted $e=|\alpha|$. The objects of $\cc$ together with the homogenous morphisms of degree 1 form a \kt linear subcategory of $\cc$ called the \emph{1-subcategory} of $\cc$ and denoted $\cc^1$.

Given $e \in E$, by an \emph{$e$-isomorphism} we mean an isomorphism which is homogeneous of degree~$e$. We say that an object $X$ is \emph{$e$-isomorphic} to an object $Y$ if there is an $e$-isomorphism $X \to Y$.

Given $e \in E$, an object~$D$ of~$\cc$ is an \emph{$e$-direct sum} of a finite family $(X_a)_{a \in A}$ of objects of~$\cc$ if there is a family $(p_a, q_a)_{a \in A}$ of morphisms such that for all~$a,b \in A$,  $p_a\co D \to X_a$ is homogeneous of degree $e^{-1}$,  $q_a \co X_a \to D$ is homogeneous of degree~$e$, $p_a q_b=\delta_{a,b} \id_{X_a}$, and $\id_D=\sum_{a \in A} q_a p_a$.
Such an $e$-direct sum $D$, if it exists, is  unique up to a 1-isomorphism and is denoted  by
$$
D=  \bigoplus_{a \in A}^e X_a.
$$
Note that for any finite families $(X_a)_{a \in A}$  and $(Y_b)_{b \in B}$  of objects of~$\cc$ and for any $d,e,f \in E$, there are \kt linear isomorphisms
\begin{equation}\label{eq-Hom-direct-sum}
\Hom_\cc^d \left( \bigoplus_{a \in A}^e X_a,\bigoplus_{b \in B}^f Y_b \right )  \cong
\bigoplus_{\substack{a \in A \\ b \in B}} \Hom_\cc^{f^{-1}de}\bigl(X_a,Y_b\bigr).
\end{equation}
By definition, a direct sum of an empty family of objects of a \kt linear category $\cc$ is a \emph{zero object}   of $\cc$, that is, an object $\mathbf{0}$ of $ \cc$ such that $\End_\cc(\mathbf{0})=0$.

An $E$-Hom-graded category~$\cc$ is \emph{$E$-additive} if any finite (possibly empty) family of objects of~$\cc$ has an $e$-direct sum in~$\cc$ for all $e \in E$.

\subsection{Monoidal categories graded by a crossed module}\label{sect-monoidal-graded-crossed-modules}
Let $\CM \co E \to H$ be a crossed module.
A \emph{$\CM$-graded monoidal category (over $\kk$)}, or shorter a  \emph {$\CM$-category}, is a $\kk$-linear  mo\-noi\-dal category $\cc=(\cc,\otimes,\un)$ such that:
\begin{enumerate}
  \labela
  \item  The $\kk$-linear category $\cc$ is $E$-Hom-graded (see Section~\ref{sect-Hom-graded-categories}).
  \item  The associativity constraints $(X\otimes Y)\otimes Z\cong
X\otimes (Y\otimes Z)$ and the unitality constraints $X\otimes \un
\cong X\cong \un \otimes X$ of $\cc$ are all homogenous of degree $1 \in E$.
\item The 1-subcategory~$\cc^1$ of $\cc$ (endowed with the monoidal structure induced by $\cc$) is $H$-graded (see Section~\ref{sect-monoidal-graded-groups}).
\end{enumerate}
These data should satisfy two conditions relating the degree of objects and morphisms. To express them, we say that an object $X$ of $\cc$ is \emph{homogeneous} if it is homogeneous in $\cc^1$ (see Section~\ref{sect-monoidal-graded-groups}). Recall that the degree of a homogeneous object $X$ is denoted by $|X| \in H$ and that the degree of a nonzero homogeneous morphism $\alpha$ is denoted by $|\alpha| \in E$. The two conditions are:
\begin{enumerate}
  \labela
  \setcounter{enumi}{3}
  \item  For all homogenous objects $X,Y$ and $e \in E$ such that $|Y|\neq \CM(e)|X|$,
  $$
  \Hom_\cc^e(X,Y)=0.
  $$
  \item The monoidal product $\alpha \otimes \beta$ of any two nonzero homogeneous morphisms $\alpha,\beta$ is a homogeneous morphism of degree
$$
|\alpha\otimes \beta|=|\alpha|\lact{|s(\alpha)|}{|\beta|},
$$
whenever the source $s(\alpha)$ of $\alpha$ is a homogeneous object. In other words, for any objects $X,Y,Z,T$ with $X$ homogeneous  and for any morphisms $\alpha \in \Hom_\cc^e(X,Y)$, $\beta \in \Hom_\cc^f(Z,T)$ with $e,f \in E$, we have:
$$
\alpha \otimes \beta \in  \Hom_\cc^{e\lact{|X|}{\!f}}(X\otimes Z,Y \otimes T).
$$
\end{enumerate}

Observe that this definition of a $\CM$-category is equivalent to the definition given in~\cite{SV} (by taking $\cch$ to be the class of homogenous objects of $\cc$ and the degree map $|\cdot| \co\cch \to H$ to be the degree of homogeneous objects).

Note that the convention of Section~\ref{sect-conventions} remains valid since Axiom (b) implies that the suppression of the associativity  and unitality  constraints does not change the degree of morphisms. Also, it follows from the axioms of an $H$-graded category that each object of $\cc$ is a $1$-direct sum of a finite family of homogeneous objects. In particular, the Hom-sets in $\cc$ are fully determined by the Hom-sets between homogeneous objects. 
Axiom (d) implies that for any homogenous objects $X,Y$,  %we have:
$$
\Hom_{\cc}(X,Y)= \hspace{-1.8em} \bigoplus_{e \in \CM^{-1}(|Y||X|^{-1})} \hspace{-1.8em} \Hom_\cc^e(X,Y) \quad \text{and} \quad \End_{\cc}(X)=\!\!\bigoplus_{e \in \Ker(\CM)} \!\!\End_\cc^e(X).
$$
In particular $\Hom_\cc(X,Y)=0$ whenever $|Y||X|^{-1} \not  \in \Ima(\CM)$.
Also, if $\alpha \co X \to Y$ is a nonzero homogeneous morphism between homogeneous objects, then
$$
|Y|=\CM(|\alpha|)\,|X|.
$$
In particular, 1-isomorphic homogenous  objects have the same degree.

A $\CM$-category $\cc$ is \emph{closed} if it is closed as a monoidal category (see Section~\ref{sect-closed}) and all (co)evaluations are homogenous of degree $1 \in E$. Note that a $\CM$-category $\cc$ is closed if and only if its 1-subcategory $\cc^1$ is closed.

A $\CM$-category $\cc$ is \emph{rigid} if it is rigid as a monoidal category (see Section~\ref{sect-rigid}) and all (co)evaluations are homogenous of degree $1 \in E$. Note that a $\CM$-category $\cc$ is rigid if and only if its 1-subcategory $\cc^1$ is rigid.

A $\CM$-category $\cc$  is \emph{pivotal} if it is endowed with a pivotal structure (see Section~\ref{sec-pivot}) such that the dual $X^*$ of any homogenous object $X$ is a homogenous object of degree $|X^*|=|X|^{-1}$ and all (co)evaluations  are homogenous of degree $1 \in E$.
Note that a $\CM$-category $\cc$ is pivotal if and only if its 1-subcategory $\cc^1$ is pivotal.

\subsection{Example}\label{ex-kk-GXi}
The linearization $\kk\tg_\CM$ of the groupoid $\tg_\CM$ associated with a crossed module $\CM\co E \to H$ (see Section~\ref{sect-crossed-modules-as-Hopf-algebras}) is a $\CM$-graded monoidal category. The objects of $\kk\tg_\CM$ are the elements of~$H$. Each $x \in H$ is homogeneous with degree $|x|=x$.
For any $x,y \in H$ and $e\in E$,
%, the set $\Hom_{\kk\tg_\CM}(x,y)$ is the free \kt module with basis $\Hom_{\tg_\CM}(x,y)=\{e \in E \, | \, y=\CM(e)x\}$ and decomposes as
$$
\Hom_{\kk\tg_\CM}^e(x,y)=\left\{ \begin{array}{ll} \kk e & \text{if $y=\CM(e)x$,} \\ 0 & \text{otherwise.} \end{array} \right.
$$
The composition of morphisms is induced by the product of $E$. The monoidal product of $\kk\tg_\CM$ is the linear extension of the product of the groupoid $\tg_\CM$:
$$
x \otimes y=xy \quad \text{and} \quad \left(x \xrightarrow{e} y\right) \otimes \left(z \xrightarrow{f} t \right) = \left(xz \xrightarrow{e\lact{x}{\!f}} yt \right).
$$
The $\CM$-category $\tg_\CM$  is pivotal. Its pivotal structures are parameterized by the group homomorphisms $d \co H \to \kk^*$. Given such a $d$, the dual of an object $x \in H$ is $x^*=x^{-1}$ with left and right (co)evaluations given by
$$
\lev_x=1, \quad \lcoev_x=1, \quad \rev_x=d(x) 1, \quad \rcoev_x=d(x)^{-1} 1,
$$
where  $1$ is the unit element of $E$.

\subsection{Particular cases}\label{sect-part-cases}
1. Given a group $H$, the trivial map $1 \to H$ is a crossed module and the notion of a $(1\to H)$-category agrees with that of an $H$-graded monoidal category.

2. Given an abelian group $E$, the trivial map $E \to 1$ is a crossed module and the notion of an $(E\to 1)$-category  agrees with that of a \kt linear monoidal category such that the Hom-sets are $E$-graded \kt modules, the composition and monoidal product of morphisms are
$E$-multiplicative (i.e., multiplicative with respect to the degree of morphisms),  and the identities and monoidal constraints are of degree $1 \in E$. In other words, $(E\to 1)$-categories are monoidal categories enriched over the symmetric category of $E$-graded \kt modules.

\section{Hopf group-coalgebras and their modules}\label{sect-Hopf-group-coalgebra}
Throughout this section, $H$ denotes a group with neutral element $1$. We review the definitions and basic properties of Hopf $H$-coalgebras and their modules (referring to \cite{Vi2} for details).

\subsection{Group-coalgebras}\label{sect-label-H-coalgebras}
An \emph{$H$-coalgebra} (over $\kk$) is a family $A=\{A_x\}_{x \in H}$ of \kt modules endowed with a family $\Delta=\{\Delta_{x,y}\co A_{xy} \to A_x \otimes A_y \}_{x,y \in H}$ of \kt linear homomorphisms (the \emph{coproduct}) and a \kt linear homomorphism $\varepsilon\co A_1 \to \kk$ (the \emph{counit}) which are coassociative and counital in the sense that for all $x,y,z \in H$,
$$
(\Delta_{x,y}\otimes \id_{A_z}) \Delta_{xy,z}=(\id_{A_x} \otimes \Delta_{y,z}) \Delta_{x,yz}
$$
and
$$
(\id_{A_x} \otimes \varepsilon) \Delta_{x,1}=\id_{A_x}=(\varepsilon \otimes \id_{A_x}) \Delta_{1,x}.
$$
Note  that $(A_1,\Delta_{1,1},\varepsilon)$  is a coalgebra (over $\kk$) in the usual sense.

To any $H$-coalgebra $A=(\{A_x\}_{x \in H}, \Delta,\varepsilon)$ and any \kt algebra $B$, one associates the $H$-graded \kt algebra
$$
\conv(A,B)=\bigoplus_{x \in H}\Hom_\kk(A_x,B),
$$
called the \emph{convolution algebra}, whose unit is $\varepsilon 1_B$ and whose product is defined by
$$
 f*g=m(f \otimes g) \Delta_{x,y} \in \Hom_\kk(A_{xy},B)
$$
for all $f \in \Hom_\kk(A_x,B)$ and $g \in \Hom_\kk(A_y,B)$, where $m$ and $1_B$ are the product and unit of $B$.

\subsection{Group-bicoalgebras}\label{sect-bi-H-coalgebras}
An \emph{$H$-bicoalgebra} (over $\Bbbk$) is an $H$-coalgebra (over $\kk$)
$
A=(\{A_x\}_{x \in H}, \Delta,\varepsilon)
$
such that each $A_x$ is endowed with a structure of a \kt algebra %(with multiplication $m_x$ and unit element $1_x \in A_x$)
so that $\varepsilon$ and~$\Delta_{x,y}$ are algebra homomorphisms for all $x,y \in H$,  that is,
\begin{align*}
&\Delta_{x,y} \,\mu_{xy}=(\mu_x\otimes \mu_y)(\id_{A_x} \otimes \sigma_{A_y ,A_x} \otimes \id_{A_y})(\Delta_{x,y} \otimes \Delta_{x,y}),
&& \varepsilon \mu_1=\varepsilon \otimes \varepsilon, \\
& \Delta_{x,y}(1_{xy})=1_x \otimes 1_y,  && \varepsilon (1_1)=1_\kk,
\end{align*}
where $\mu_x\co A_x \otimes A_x \to A_x$ and $1_x \in A_x$ denote the product and the unit element of~$A_x$. Here and below,  for \kt modules $M$ and $N$, the flip $\sigma_{M,N}\co M \otimes N \to N \otimes M$ is defined by
$\sigma_{M,N}(u \otimes v)=v \otimes u$ for all $u \in M$ and $v \in N$.

\subsection{Antipodes}\label{sect-Hopf-H-coalgebras-antipodes}
An \emph{antipode} for an $H$-bicoalgebra $A=(\{A_x\}_{x \in H}, \Delta,\varepsilon)$ is a family $S=\{S_x \co  A_{x^{-1}} \to A_x \}_{x \in H}$ of \kt linear homomorphisms such that for all $x \in H$,
$$
\mu_x (S_x \otimes \id_{A_x}) \Delta_{x^{-1},x} =   \eta_x \varepsilon  =  \mu_x (\id_{A_x} \otimes S_x) \Delta_{x,x^{-1}}
$$
where $\mu_x$ and $\eta_x=(1_\kk \in \kk \mapsto 1_x \in A_x)$ are the product and the  unit map of $A_x$. This axiom means that for any $x \in H$, the homomorphism $S_x$ is the inverse of $\id_{A_x}$ in the convolution algebra $\conv(A,A_x)$. An antipode $S=\{S_x\}_{x \in H}$ is \emph{bijective} if~$S_x$ is bijective for all $x \in H$.

If it exists, an antipode is unique. Also, it is anti-multiplicative: for all $x \in H$,
$$
S_x \mu_{x^{-1}}=\mu_x\sigma_{A_x,A_x}(S_x \otimes S_x) \quad \text{and} \quad S_x(1_{x^{-1}})=1_x,
$$
and anti-comultiplicative:  for all $x,y \in H$,
$$
\Delta_{x,y} S_{xy}=(S_x \otimes S_y)\sigma_{A_{y^{-1}},A_{x^{-1}}} \Delta_{y^{-1},x^{-1}} \quad \text{and} \quad  \varepsilon S_1=\varepsilon.
$$

\subsection{Hopf group-coalgebras}\label{sect-Hopf-H-coalgebras}
A \emph{Hopf $H$-coalgebra} is an $H$-bicoalgebra endowed with a bijective antipode. When $H=1$, one recovers the usual notion of a Hopf algebra. In particular, if~$A$ is a Hopf $H$-coalgebra, then  $A_1$ is a Hopf algebra.

The product $\mu_x$, unit map $\eta_x$, %$\eta_x=(1_\kk \in \kk \mapsto 1_x \in A_x)$,
coproduct $\Delta_{x,y}$,  counit $\varepsilon$,  antipode $S_x$ and its inverse $S_x^{-1}$ of a Hopf $H$-coalgebra $A=\{A_x\}_{x \in H}$ are depicted as follows:
$$
\psfrag{a}[Br][Br]{\scalebox{.8}{$x$}}
\psfrag{x}[Bl][Bl]{\scalebox{.8}{$x$}}
\psfrag{y}[Br][Br]{\scalebox{.8}{$y$}}
\psfrag{n}[Bl][Bl]{\scalebox{.8}{$y$}}
\psfrag{z}[Bl][Bl]{\scalebox{.8}{$xy$}}
\psfrag{u}[Bl][Bl]{\scalebox{.8}{$1$}}
\psfrag{c}[Bl][Bl]{\scalebox{.8}{$x^{-1}$}}
\mu_x=\!\rsdraw{.45}{.9}{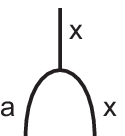} \quad
\eta_x=\,\rsdraw{.45}{.9}{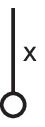} \quad
\Delta_{x,y}=\!\rsdraw{.45}{.9}{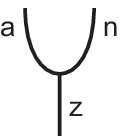}\quad
\varepsilon=\,\rsdraw{.45}{.9}{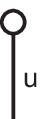} \quad
S_x=\;\,\rsdraw{.45}{.9}{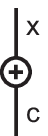}\quad\;
\psfrag{x}[Bl][Bl]{\scalebox{.8}{$x^{-1}$}}
\psfrag{c}[Bl][Bl]{\scalebox{.8}{$x$}}
S_{x}^{-1}=\;\,\rsdraw{.45}{.9}{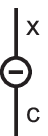}\;.
$$
Here the colors $x,y \in H$ are abbreviations for $A_x$ and $A_y$. The axioms of a Hopf $H$-coalgebra are then easily depicted in a manner similar to that of Sections~\ref{sect-categorical-algebras}-\ref{sect-categorical-Hopf-algebras}.

An \emph{$H$-grouplike element} of a Hopf $H$-coalgebra $A$ is a family $G=(G_x)_{x \in H}$ with $G_x \in A_x$ such that for all $x,y \in H$,
$$
\Delta_{x,y}(G_{xy})=G_x \otimes G_y \quad \text{and} \quad \epsilon(G_1)=1_\kk.
$$
Note that each $G_x$ is then invertible in $A_x$ with inverse 
$$
G_x^{-1}=S_x(G_{x^{-1}})=S_{x^{-1}}^{-1}(G_{x^{-1}}),
$$
where $S=\{S_x\}_{x \in H}$ is the antipode of $A$. 
The set $G_H(A)$ of $H$-grouplike elements of $A$ is thus a group for the pointwise product.

\subsection{Modules over Hopf group-coalgebras}\label{sect-modules-coalgebras}
Any $H$-bicoalgebra $A=\{A_x\}_{x \in H}$ (over $\kk$) yields the $H$-graded \kt linear monoidal category $\Mod_H(A)$ of $A$-modules and $A$-linear morphisms defined as follows.

A \emph{(left) $A$-module} is an $H$-graded \kt module $$M=\bigoplus_{x \in H} M_x$$ such that each $M_x$ is endowed with a structure of a (left) module over the \kt algebra~$A_x$. 

An \emph{$(H,A)$-linear morphism} between two $A$-modules  $M,N$ is a \kt linear  homo\-mor\-phism $\alpha \co M \to N$ such that:
\begin{enumerate}
\labela
\item  The map $\alpha$ preserves the $H$-grading: $\alpha (M_x) \subset N_x$ for all $x \in H$.
\item For all $x \in H$, the map $m \in M_x \mapsto \alpha (m) \in N_x$ is $A_x$\ti linear.
\end{enumerate}

We let $\Mod_H(A)$ be the category of $A$-modules and $(H,A)$-linear morphisms, with composition induced in the obvious way from the set-theoretical composition. 
The monoidal product of two $A$-modules $M$ and $N$ is the $A$-module $$M \otimes N=\bigoplus_{x \in H} (M \otimes N)_x \quad \text{where} \quad 
(M \otimes N)_x=\bigoplus_{\substack{y,z \in H\\ yz=x}} M_y \otimes N_z$$
is endowed with the $A_x$-action defined on $M_y \otimes N_z$ by
$$
\psfrag{X}[Bl][Bl]{\scalebox{.8}{$M_y$}}
\psfrag{Y}[Bl][Bl]{\scalebox{.8}{$N_z$}}
\psfrag{x}[Br][Br]{\scalebox{.8}{$x$}}
\psfrag{y}[Br][Br]{\scalebox{.8}{$y$}}
\psfrag{z}[Br][Br]{\scalebox{.8}{$z$}}
\rsdraw{.45}{.9}{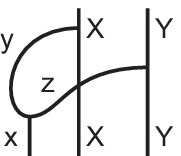} \;.
$$
The monoidal product of $(H,A)$-linear morphisms is their tensor product over $\kk$.  The unit object $\un$ is $\kk$ concentrated in degree $1 \in H$ with action given by the counit $\varepsilon \co A_1 \to \kk$.  The monoidal constraints of $\Mod_H(A)$ are inherited from the category of \kt modules. Then $\Mod_H(A)$ is a \kt linear monoidal category. Note that the forgetful functor $\Mod_H(A) \to \Mod_\kk$ is strict monoidal. 

For any $x \in H$, by viewing any $A_x$-module as an $A$-module concentrated in degree $x$, the category 
$\Mod_\kk(A_x)=\Mod_{\Mod_\kk}(A_x)$ of $A_x$-modules (see Section~\ref{sect-modules-in-cat})
is a full subcategory of $\Mod_H(A)$. Then the category $\Mod_H(A)$ is $H$-graded by the family $\{\Mod_\kk(A_x)\}_{x \in H}$:
$$
\Mod_H(A)=\bigoplus_{x \in H} \Mod_\kk(A_x)
$$
In particular, the homogenous objects of $\Mod_H(A)$ are the nonzero $A_x$-modules where $x$ runs over $H$. Note that $\Mod_H(A)$ is additive: any finite direct sum of $A$-modules always exists (it is induced in the obvious way from the direct sums in~$\Mod_\kk$). Furthermore $\Mod_H(A)$ is abelian (with kernels and cokernels induced in the obvious way from those in $\Mod_\kk$). 

Let $\mmod_H(A)$ be the full subcategory of  $\Mod_H(A)$ consisting of the $A$-modules whose underlying \kt module is projective of finite rank. Then $\mmod_H(A)$ is a \kt linear monoidal subcategory of  $\Mod_H(A)$. It is additive and $H$-graded:
$$
\mmod_H(A)=\bigoplus_{x \in H} \mmod_\kk(A_x).
$$
where $\mmod_\kk(A_x)$ is the full subcategory of  $\Mod_\kk(A_x)$  consisting of the $A_x$\ti modules whose underlying \kt module is projective of finite rank.

If $A$ is a Hopf $H$-coalgebra, then the monoidal category $\Mod_H(A)$ is closed (so that the forgetful functor $\Mod_H(A) \to \Mod_\kk$ preserves the internal Homs) and the monoidal category $\mmod_H(A)$ is rigid. The formulas for the internal Homs and the duals of objects are similar (with obvious changes) to those given in Section~\ref{sect-modules-in-cat}.

Assume that $A$ is a Hopf $H$-coalgebra. A \emph{pivotal element} for $A$ is an $H$-grouplike element $G=(G_x)_{x \in H}$ of $A$ (see Section~\ref{sect-Hopf-H-coalgebras}) such that for all $x \in H$ and $a \in A_x$,
$$
S_x S_{x^{-1}} (a)=G_x a G_x^{-1}.
$$
Then pivotal structures on $\mmod_H(A)$ are in bijective correspondence with pivotal elements of $A$. The pivotal structure on $\mmod_H(A)$ associated with a  pivotal element $G=(G_x)_{x \in H}$ is given as follows: the dual of an object $(M,r) \in \mmod_\kk(A_x)$ is
the object $(M,r)^*=(M^*,r^\dagger)\in \mmod_\kk(A_{x^{-1}})$, where $M^*=\Hom_\kk(M,\kk)$ and the action~$r^\dagger$ is given as in  Section~\ref{sect-modules-in-cat}: for all $a \in A_{x^{-1}}$,  $ \varphi \in M^\ast$, and $m \in M$, 
$$
r^\dagger(a \otimes \varphi)(m)=\varphi \bigl(S_x(a)m\bigr).
$$
The associated (co)evaluations are given, for all $\varphi \in M^\ast$, $m \in M$, by
\begin{align*}
&\lev_{(M,r)}(\varphi \otimes m)=\varphi(m) ,  && \rev_M(m \otimes \varphi)= \varphi(G_xm) , \\
& \lcoev_{(M,r)}(1_\kk)=\sum_{i} b_i \otimes b_i^*,
 && \rcoev_{(M,r)}(1_\kk)=\sum_{i} b_i^* \otimes G_x^{-1}b_i,
\end{align*}
where  $(b_i)_{i}$ is  any   basis of $M$ and  $(b_i^*)_{i}$ is the dual basis of $M^*$.

\subsection{Example}\label{ex-kk-H}
Consider the trivial Hopf $H$-coalgebra $\kk_H= \{(\kk_H)_x=\kk\}_{x \in H}$ whose structural morphisms are given for all $x,y \in H$ by
$$
\Delta_{x,y}(1_{\kk})= 1_{\kk} \otimes 1_{\kk}, \quad \varepsilon= \id_{\kk}, \quad S_x = \id_{\kk}.
$$
Then the closed $H$-graded monoidal category $\Mod_H(\kk_H)$ is nothing but the category of $H$-graded \kt modules and  \kt linear  grading-preser\-ving homomorphisms.

\subsection{Remark}\label{rem-morph-H-coalgebras}
A homomorphism from an $H$-bicoalgebra $A=\{A_x\}_{x \in H}$ to an $H$\ti bicoalgebra $B=\{B_x\}_{x \in H}$ is a family $f=\{f_x \co A_x \to B_x\}_{x\in H}$ of algebra homomorphisms compatible with the coproducts and counits of $A$ and $B$ in the following sense: for all $x,y \in H$,
$$
\Delta_{x,y}^{\! B} f_{xy}=(f_x \otimes f_y)\Delta^{\! A}_{x,y} \quad \text{and} \quad \varepsilon^B f_1=\varepsilon^A.
$$
Note that if $A$ and $B$ are Hopf $H$-coalgebras, then any $H$-bicoalgebra homomorphism $f\co A \to B$ preserves the antipodes of $A$ and $B$, that is,  $S_{x}^Bf_{x^{-1}}= f_x S_x^A$ for all $x \in H$.

It is not difficult to check that any $H$-bicoalgebra homomorphism $f\co A \to B$ induces an $H$-graded functor $f^*\co \Mod_H(B) \to \Mod_H(A)$ that is, a strong monoidal \kt linear  functor that preserves the $H$-grading of objects. Moreover, if $A$ and $B$ are Hopf $H$-coalgebras, then the functor $f^*$ is closed, that is,  it preserves the internal Homs. Also, if $A$ and $B$ are Hopf $H$-coalgebras endowed with pivotal elements, then any $H$-bicoalgebra homomorphism $f\co A \to B$ preserving the pivotal elements of $A$ and $B$ (that is, $f_x(G_x^A)=G^B_x$ for all $x \in H$) induces a pivotal $H$-graded functor $f^*\co \mmod_H(B) \to \mmod_H(A)$.

\section{Hopf crossed module-coalgebras}\label{sect-Hopf-gcrossed-modules-coalgebra}
Throughout this section, $\CM \co E \to H$ is a crossed module. We introduce the notions of $\CM$-bicoalgebras and Hopf $\CM$-coalgebras. The main motivation is that their categories of representations (introduced in Section~\ref{sect-representations-Hopf-Xi-coalgebras}) are $\CM$-graded monoidal categories.

\subsection{Crossed module-actions}\label{sect-crossed-module-actions}
A \emph{$\CM$-action} on an $H$-bicoalgebra $A=\{A_x\}_{x \in H}$ is a family
$$
\phi=\{\phi_{x,e} \co A_x \to A_{\CM(e)x} \}_{(x,e) \in H \times E}
$$
of \kt algebra homomorphisms such that for all $x,y \in H$ and $e,f \in E$,
\begin{align}
%& \text{$\phi_{x,e}$ is an algebra morphism,} \label{eq-Xi-act-1}\\
& \phi_{x,1}=\id_{A_x}, \label{eq-Xi-act-1} \\
& \phi_{\CM(e)x,f}\, \phi_{x,e}=\phi_{x,fe}, \label{eq-Xi-act-2} \\
& (\phi_{x,e} \otimes \phi_{y,f}) \Delta_{x,y}=\Delta_{\CM(e)x,\CM(f)y}\, \phi_{xy,e\lact{x}{\!f}}. \label{eq-Xi-act-3}
\end{align}
Note the indices in Axiom~\eqref{eq-Xi-act-3} are coherent since the equivariance of $\CM$ (see Section~\ref{sect-crossed-modules-def}) implies that
$\CM\bigl(e\lact{x}{\!f}\bigr)xy=\CM(e)x\CM(f)y$. Also, it follows directly from \eqref{eq-Xi-act-1} and \eqref{eq-Xi-act-2}  that each $\phi_{x,e}$ is an isomorphism and
$$
\phi_{x,e}^{-1}=\phi_{\CM(e)x, e^{-1}}.
$$
A $\CM$-action is \emph{trivial} if $\phi_{x,e}=\id_{A_x}$ (and so $A_{\CM(e)x}=A_x$) for all $x \in H$ and $e \in E$.

We depict the $\CM$-action $\phi_{x,e} \co A_x \to A_{\CM(e)x}$ by a strand with a dot labeled with $e$ (on the left or on the right) as follows:
$$
\psfrag{x}[Bl][Bl]{\scalebox{.8}{$x$}}
\psfrag{y}[Bl][Bl]{\scalebox{.8}{$\CM(e)x$}}
\psfrag{e}[Br][Br]{\scalebox{.9}{$e$}}
\phi_{x,e}= \, \rsdraw{.45}{.9}{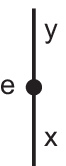} \quad \text{or} \quad
\psfrag{e}[Bl][Bl]{\scalebox{.9}{$e$}}
 \phi_{x,e}= \,\rsdraw{.45}{.9}{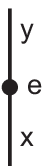} \;.
$$
Axioms~\eqref{eq-Xi-act-1}-\eqref{eq-Xi-act-3}  are  depicted as
$$
\psfrag{x}[Bl][Bl]{\scalebox{.8}{$x$}}
\psfrag{e}[Br][Br]{\scalebox{.9}{$1$}}
\rsdraw{.45}{.9}{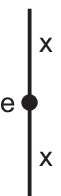}= \,
\psfrag{x}[Bl][Bl]{\scalebox{.8}{$x$}}
\rsdraw{.45}{.9}{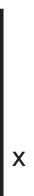} \;, \quad \quad \qquad
\psfrag{x}[Br][Br]{\scalebox{.8}{$x$}}
\psfrag{y}[Br][Br]{\scalebox{.8}{$\CM(e)x$}}
\psfrag{z}[Br][Br]{\scalebox{.8}{$\CM(f)\CM(e)x$}}
\psfrag{u}[Bl][Bl]{\scalebox{.9}{$f$}}
\psfrag{e}[Bl][Bl]{\scalebox{.9}{$e$}}
\rsdraw{.45}{.9}{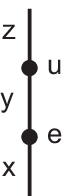}= \;
\psfrag{x}[Bl][Bl]{\scalebox{.8}{$x$}}
\psfrag{a}[Bl][Bl]{\scalebox{.8}{$\CM(fe)x$}}
\psfrag{v}[Bl][Bl]{\scalebox{.9}{$fe$}}
\rsdraw{.45}{.9}{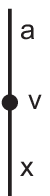} \quad, \qquad \qquad
\psfrag{z}[Bl][Bl]{\scalebox{.8}{$xy$}}
\psfrag{n}[Bl][Bl]{\scalebox{.8}{$\CM\bigl(e\lact{x}{\!f}\bigr)xy$}}
\psfrag{a}[Bl][Bl]{\scalebox{.8}{$\CM(f)y$}}
\psfrag{c}[Br][Br]{\scalebox{.8}{$\CM(e)x$}}
\psfrag{x}[Br][Br]{\scalebox{.8}{$x$}}
\psfrag{y}[Bl][Bl]{\scalebox{.8}{$y$}}
\psfrag{u}[Bl][Bl]{\scalebox{.9}{$f$}}
\psfrag{e}[Br][Br]{\scalebox{.9}{$e$}}
\psfrag{v}[Br][Br]{\scalebox{.9}{$e\lact{x}{\!f}$}}
\rsdraw{.45}{.9}{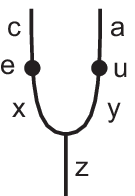} \quad \; = \; \quad \rsdraw{.45}{.9}{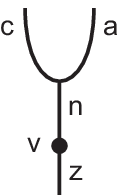} \quad \quad.
$$
The fact that $\phi_{x,e}$ is an algebra homomorphism is depicted as
$$
\psfrag{x}[Bl][Bl]{\scalebox{.8}{$x$}}
\psfrag{u}[Bl][Bl]{\scalebox{.9}{$e$}}
\psfrag{z}[Bl][Bl]{\scalebox{.8}{$\CM(e)x$}}
\psfrag{c}[Br][Br]{\scalebox{.8}{$x$}}
\psfrag{a}[Br][Br]{\scalebox{.8}{$\CM(e)x$}}
\psfrag{e}[Br][Br]{\scalebox{.9}{$e$}}
\rsdraw{.45}{.9}{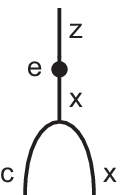}\;\, = \;\,\quad \rsdraw{.45}{.9}{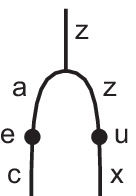}\quad \qquad \text{and} \qquad \quad
\psfrag{x}[Br][Br]{\scalebox{.8}{$x$}}
\psfrag{y}[Br][Br]{\scalebox{.8}{$\CM(e)x$}}
\psfrag{z}[Bl][Bl]{\scalebox{.8}{$\CM(e)x$}}
\psfrag{e}[Bl][Bl]{\scalebox{.9}{$e$}}
\rsdraw{.45}{.9}{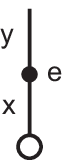}\; = \;\, \rsdraw{.45}{.9}{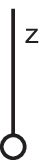} \quad \;.
$$

\subsection{Hopf crossed module-coalgebras}\label{sect-Hopf-crossed-module-coalgebras}
A  \emph{$\CM$-bicoalgebra} (over $\kk$) is an $H$-bicoal\-gebra (over $\kk$) endowed with a $\CM$-action.

A \emph{Hopf $\CM$-coalgebra} (over $\kk$)  is a Hopf $H$-coalgebra  (over $\kk$) endowed with a $\CM$-action. Equivalently,  a Hopf $\CM$-coalgebra is a $\CM$-bicoalgebra whose underlying $H$-bicoalgebra has a bijective antipode.

In Section~\ref{sect-Hopf-Xi-as-Hopf}, we prove that Hopf crossed module-coalgebras may be seen as Hopf algebras in some symmetric monoidal category (see Corollary~\ref{cor-characterization-Hopf-Xi-coalgebras}). 
In the next lemma, we show that the antipode of a Hopf $\CM$-coalgebra is compatible with the $\CM$-action.
\begin{lem}\label{lem-Xi-antipode}
Let $A=\{A_x\}_{x \in H}$ be a Hopf $\CM$-coalgebra, with antipode $S=\{S_x\}_{x \in H}$ and $\CM$-action $\phi=\{\phi_{x,e} \}_{(x,e) \in H \times E}$. Then
$$
\phi_{x,e} \, S_x = S_{\CM(e)x} \, \phi_{x^{-1},\lact{{x^{-1}}}{\!(e^{-1})}}
$$
for all $x \in H$ and $e \in E$.
\end{lem}
\begin{proof}
Axiom \eqref{eq-Xi-act-3} and the multiplicativity of $\phi_{x,e}^{-1}=\phi_{\CM(e)x, e^{-1}}$ imply that $$\phi_{x,e}^{-1}\, S_{\CM(e)x} \, \phi_{x^{-1},\lact{{x^{-1}}}{\!(e^{-1})}} \co A_{x^{-1}} \to A_x$$ is inverse to $\id_{A_x}$ in the convolution algebra $\conv(A,A_x)$, and so is equal to $S_x$.
\end{proof}
Graphically, the compatibility of the antipode with  the $\CM$-action (see Lemma~\ref{lem-Xi-antipode})  is depicted as
$$
\psfrag{e}[Br][Br]{\scalebox{.9}{$e$}}
\psfrag{u}[Br][Br]{\scalebox{.9}{$\lact{{x^{-1}}}{\!(e^{-1})}$}}
\psfrag{z}[Bl][Bl]{\scalebox{.8}{$\CM(e)x$}}
\psfrag{x}[Bl][Bl]{\scalebox{.8}{$x$}}
\psfrag{c}[Bl][Bl]{\scalebox{.8}{$x^{-1}$}}
\psfrag{a}[Bl][Bl]{\scalebox{.8}{$x^{-1}\CM(e^{-1})$}}
\rsdraw{.45}{.9}{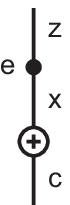}\;= \;\,
\rsdraw{.45}{.9}{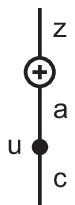}\; \quad \quad  \quad \;.
$$

\subsection{Particular cases}
1. Given a group $H$, the trivial map $1 \to H$ is a crossed module and it follows from \eqref{eq-Xi-act-1} that any $(1\to H)$-action for an $H$-bicoalgebra is trivial. Consequently, the notion of a Hopf $(1\to H)$-coalgebra  agrees with that of a Hopf $H$-coalgebra.  

2. Given an abelian group $E$, the trivial map $E \to 1$ is a crossed module and the notion of a Hopf $(E\to 1)$-coalgebra  agrees with that of a Hopf algebra $A$ endowed with an action of $E$ by algebra and bicomodule automorphisms, that is, with a group homomorphism $\phi \co E \to \Aut_\kk(A)$
such that for all $e \in E$,
\begin{itemize}
\item  $\phi_e$ is an algebra automorphism of $A$,
\item   $\phi_e$ is a bicomodule automorphism of $A$:
$$
\Delta \, \phi_{e}=(\phi_e \otimes \id_A)\Delta=(\id_A \otimes \phi_e)\Delta.
$$
\end{itemize}
Here $\Aut_{\kk}(A)$ denotes the group of \kt linear automorphisms of $A$ and $\Delta$ is the coproduct of $A$.

\subsection{Example}\label{ex-Hopf-xi-from-bicharacter}
Let $\omega \co E \times G \to \kk^*$ be a bicharacter, where $E$ is an abelian group and $G$ is a group.
Recall that the group algebra $\kk[G]$ is a Hopf algebra with copro\-duct $\Delta$, counit $\varepsilon$, and antipode $S$ defined by
$\Delta(g)=g \otimes g$, $\varepsilon(g)=1_\kk$, and~$S(g)=g^{-1}$ for all $g \in G$. Consider the group homomorphism  $\phi \co E \to \Aut_{\kk}(\kk[G])$ defined by $\phi_e(g)=\omega(e,g)  g$ for all $e \in E$ and $g \in G$. Then $\kk[G]$ is a Hopf $(E \to 1)$-coalgebra with $(E \to 1)$-action $\phi$, which we denote by $\kk^\omega[G]$.

\subsection{Example}\label{ex-Hopf-xi-from-action}
Let $\CM \co E \to H$ be a crossed module, $A$ be a Hopf \kt algebra, and $\rho \co H \to \Aut_{\alg}(A)$ be a group homomorphism, where $\Aut_\alg(A)$ is the group of algebra automorphisms of $A$. Let $\delta$,~$\varepsilon$, and $s$  be the coproduct, counit, and antipode of $A$, respectively. For any $x \in H$, set $A_x=A$ as a \kt algebra. Then the family $A_\CM^\rho=\{A_x\}_{x \in H}$ is a Hopf $\CM$-coalgebra with counit  $\varepsilon \co A_1 \to \kk$ and with coproduct, antipode, and $\CM$-action respectively defined by:
\begin{gather*}
\Delta=\{\Delta_{x,y}=(\rho_x \otimes \rho_y) \delta \rho_{(xy)^{-1}} \co A_{xy} \to A_x \otimes A_y\}_{x,y \in H},\\
S=\{S_x=\rho_x s \rho_x \co A_{x^{-1}} \to A_x\}_{x \in H},\\
\phi=\bigl\{\phi_{x,e}=\rho_{\CM(e)} \co A_x \to A_{\CM(e)x}\bigr\}_{(x,e) \in H \times E},
\end{gather*}
where $\rho_z$ denotes the image of $z \in H$ under $\rho$.

\subsection{Example}
Hopf $\CM$-coalgebras with trivial $\CM$-action are in bijective correspondence with Hopf $\pi$-coalgebras, where $\pi=\Coker(\CM)=H/\Ima(\CM)$ is the so-called fundamental group of $\CM$. Indeed, let $B=\{B_g\}_{g \in \pi}$ be a Hopf $\pi$-coalgebra.
Denote by $p \co H \to \pi$ the canonical projection. For any $x \in H$, set $A_x=B_{p(x)}$ as a \kt algebra. For any $x,y \in H$, set
$\Delta_{x,y}=\delta_{p(x),p(y)}$ and $S_x=s_{p(x)}$, where $\delta$ and $s$ are the coproduct and antipode of $B$. Then the family $A=\{A_x\}_{x \in H}$, endowed with the coproduct~$\Delta$, the same counit of $B$, and the antipode $S$, is a Hopf $\CM$-coalgebra with trivial $\CM$-action.

Conversely, let $A=\{A_x\}_{x \in H}$ be a Hopf $\CM$-coalgebra with trivial $\CM$-action. Pick a set-theoretical section $q \co \pi \to H$ of the canonical projection $p \co H \to \pi$, meaning that $pq=\id_\pi$.  Note that the triviality of the $\CM$-action implies that $A_x=A_y$ for all $x, y \in H$ such that $p(x)=p(y)$.  For any $\alpha,\beta \in \pi$, set $B_\alpha=A_{q(\alpha)}$ as a \kt algebra,
$\delta_{\alpha,\beta}=\Delta_{q(\alpha),q(\beta)}$, and $ s_\alpha=S_{q(\alpha)}$, where $\Delta$ and $S$ are the coproduct and antipode of $A$. Then the family $B=\{B_\alpha\}_{\alpha \in \pi}$, endowed with the coproduct $\delta$, the same counit of $A$, and the antipode $s$, is a Hopf $\pi$-coalgebra. Note that $B$ does not depend on the choice of the section $q$ because the $\CM$-action of $A$ is trivial.

\subsection{Grouplike elements}\label{sect-Xi-grouplike}
By an \emph{$H$-grouplike element} of a Hopf $\CM$-coalgebra,  we mean an $H$-grouplike element of its underlying Hopf $H$-coalgebra (see Section~\ref{sect-Hopf-H-coalgebras}). As in Section~\ref{sect-Hopf-H-coalgebras}, we denote by $G_H(A)$ the group of $H$-grouplike elements of $A$.

A \emph{$\CM$-grouplike element} of a Hopf $\CM$-coalgebra $A$ is an $H$-grouplike element $G=(G_x)_{x\in H}$ of $A$  which is $\CM$-equivariant in the sense that for all $x \in H$ and $e \in E$,
$$
\phi_{x,e}(G_x)=G_{\CM(e)x}.
$$
The set $G_\CM(A)$ of $\CM$-grouplike elements of $A$ is a subgroup of $G_H(A)$. Note that it may be a strict subgroup. For example, for the Hopf $(E \to 1)$-coalgebra $\kk^\omega[G]$ of Example~\ref{ex-Hopf-xi-from-bicharacter}, we have:
$$
G_1\bigl(\kk^\omega[G]\bigr)=G \quad \text{and} \quad G_{E \to 1}\bigl(\kk^\omega[G]\bigr)=  \{ g \in G \,  | \, \omega(e,g)=1_\kk \text{ for all } e \in E\}.
$$

The next lemma describes the behaviour of the $\CM$-action on $H$-grouplike elements.

\begin{lem}\label{lem-Xi-bicharacter}
Let $A$ be a Hopf $\CM$-coalgebra. Then the map
$$
G_H(A) \times E \to \kk^*, \quad (G,e)  \mapsto  \langle G,e \rangle=\varepsilon\bigl(\phi_{\CM(e^{-1}),e}(G_{\CM(e^{-1})}) \bigr)
$$
is a bicharacter such that for all $G \in G_H(A)$ and $(x,e) \in H \times E$,
$$
\phi_{x,e}(G_x)=\langle G,e \rangle\, G_{\CM(e)x}.
$$
Here, $\varepsilon$ is the counit and $\phi$ is the $\CM$-action of~$A$.
\end{lem}
It follows directly from Lemma~\ref{lem-Xi-bicharacter} that an $H$-grouplike element $G$ of $A$  is a $\CM$-grouplike element if and only if $\langle G,e \rangle=1$ for all $e \in E$.
\begin{proof}
For all $G \in G_H(A)$ and $(x,e) \in H \times E$, we have:
\begin{align*}
\phi_{x,e}(G_x)
& \overset{(i)}{=} (\varepsilon \otimes \id_{A_{\CM(e)x}})\Delta_{1,\CM(e)x}\, \phi_{x,e}(G_x) \\
& \overset{(ii)}{=} (\varepsilon\phi_{\CM(e^{-1}),e} \otimes \phi_{\CM(e)x,1})\Delta_{\CM(e^{-1}),\CM(e)x}(G_x) \\
& \overset{(iii)}{=} \varepsilon\bigl(\phi_{\CM(e^{-1}),e}(G_{\CM(e^{-1})})\bigr) \, \phi_{\CM(e)x,1}(G_{\CM(e)x}) \\
& \overset{(iv)}{=} \langle G,e \rangle\, G_{\CM(e)x}.
\end{align*}
Here $(i)$ follows the counitality of the coproduct, $(ii)$ from~\eqref{eq-Xi-act-3}, $(iii)$ from the fact that $G$ is $H$-grouplike, and $(iv)$ from~\eqref{eq-Xi-act-1} and the definition of $\langle G,e \rangle$. The multiplicativity of the map $(G,e)  \mapsto  \langle G,e \rangle$ in the variable  $G$ follows from the multiplicativity of the counit and of the maps $\phi_{\CM(e^{-1}),e}$. Its multiplicativity in the variable  $e$ follows from \eqref{eq-Xi-act-3}. Hence this map is indeed a bicharacter $G_H(A) \times E \to \kk^*$.
\end{proof}

\subsection{The dual notion: Hopf crossed module-algebras}\label{sect-dual-notion}
The notion of a Hopf $\CM$-coalgebra is not self dual. Its dual notion is that of a Hopf $\CM$-algebra defined as follows.
Recall that a \kt algebra $A$ is \emph{$H$-graded} if it is endowed with a direct sum decomposition $A=\bigoplus_{x \in H} A_x$ such that $1_A \in A_1$ and $A_xA_y \subset A_{xy}$ for all $x,y \in H$.

An \emph{$H$-bialgebra} (over $\Bbbk$) is an $H$-graded \kt algebra $A=\bigoplus_{x \in H} A_x$
such that each $A_x$ is endowed with a structure of a \kt coalgebra %(with multiplication $m_x$ and unit element $1_x \in A_x$)
so that the unit map $\eta \co \kk \to A$ (defined by $\eta(1_\kk)=1_A$) and the restricted products $\mu_{x,y}\co A_x \otimes A_y \to A_{xy}$ are coalgebra homomorphisms for all $x,y \in H$,  that is,
\begin{align*}
&\Delta_{xy} \,\mu_{x,y}=(\mu_{x,y}\otimes \mu_{x,y})(\id_{A_x} \otimes \sigma_{A_x ,A_y} \otimes \id_{A_y})(\Delta_x \otimes \Delta_y),
&& \varepsilon_{xy} \mu_{x,y}=\varepsilon_x \otimes \varepsilon_y, \\
& \Delta_1(1_A)=1_A \otimes 1_A,  && \varepsilon_1(1_A)=1_\kk,
\end{align*}
where $\Delta_x\co A_x \to A_x \otimes A_x$ and $\varepsilon_x \co A_x \to \kk$ denote the coproduct and counit of~$A_x$, and $\sigma_{A_x ,A_y}$ is the usual flip.

A \emph{Hopf $H$-algebra} is an $H$-bialgebra $A=\bigoplus_{x \in H} A_x$ endowed with a bijective antipode, that is, a family $S=\{S_x \co  A_x \to A_{x^{-1}} \}_{x \in H}$ of \kt linear isomorphisms such that for all $x \in H$,
$$
\mu_{x^{-1},x} (S_x \otimes \id_{A_x}) \Delta_x = \varepsilon_x \, 1_A  =  \mu_{x,x^{-1}} (\id_{A_x} \otimes S_x) \Delta_x.
$$

A \emph{$\CM$-action} on a $H$-bialgebra $A=\bigoplus_{x \in H} A_x$ is a family
$$
\phi=\{\phi_{x,e} \co A_x \to A_{\CM(e)x} \}_{x \in H, e \in E}
$$
of \kt coalgebra homomorphisms such that for all $x,y \in H$ and $e,f \in E$,
\begin{itemize}
\item $\phi_{x,1}=\id_{A_x}$,
\item  $\phi_{\CM(e)x,f}\, \phi_{x,e}=\phi_{x,fe}$,
\item  $\mu_{\CM(e)x,\CM(f)y}\, (\phi_{x,e} \otimes \phi_{y,f}) = \phi_{xy,e\lact{x}{\!f}} \,\mu_{x,y}$.
\end{itemize}
Note that each $\phi_{x,e}$ is then an isomorphism and $\phi_{x,e}^{-1}=\phi_{\CM(e)x, e^{-1}}$.

A  \emph{$\CM$-bialgebra} is an $H$-bialgebra endowed with a $\CM$-action. A \emph{Hopf $\CM$-algebra} is a Hopf $H$-algebra endowed with a $\CM$-action.

A Hopf $\CM$-(co)algebra $A$ is of \emph{finite type} if for all $x \in H$, the \kt module $A_x$ is projective of finite rank.

The dual of a Hopf $\CM$-algebra of finite type is a Hopf $\CM$-coalgebra (of finite type). More explicitly, let $A=\bigoplus_{x \in H} A_x$ be a Hopf $\CM$-algebra of finite type. For any \kt module $M$, set $M^*=\Hom_\kk(M,\kk)$. Then the dual $A^*=\{A^*_x\}_{x \in H}$ of $A$ is a Hopf $\CM$-coalgebra. Its coproduct $\Delta_{x,y}\co A_{xy}^* \to A_x^* \otimes A_y^*$ is induced by the transpose of the restricted product $A_x \otimes A_y \to A_{xy}$ and the canonical \kt linear isomorphism
$(A_x \otimes A_y)^* \cong A_x^* \otimes A_y^*$. The algebra structure of $A_x^*$ is induced in the standard way by the coalgebra structure of $A_x$. The antipode and $\CM$-action of $A^*$ are the transpose of those of $A$.

Conversely, the dual $A^*=\bigoplus_{x \in H} A_x^*$ of a Hopf $\CM$-coalgebra $A=\{A_x\}_{x \in H}$ of finite type is a Hopf $\CM$-algebra (with transposed structural morphisms).

\section{Categories of representations of Hopf crossed module-coalgebras}\label{sect-representations-Hopf-Xi-coalgebras}
Throughout this section, $\CM \co E \to H$ is a crossed module and $A=\{A_x\}_{x \in H}$  is a  $\CM$-bicoalgebra (over $\kk$).
We associate to  $A$ two categories of representations $\Mod_\CM(A)$ and $\mmod_\CM(A)$ which are $\CM$-graded categories (in the sense of Section~\ref{sect-monoidal-graded-crossed-modules}).

\subsection{Modules over $\CM$-bicoalgebras}\label{sect-modules-Xi-coalgebras-def}
A \emph{(left) $A$-module} is a module over the $H$\ti bicoalgebra underlying~$A$, that is, an $H$-graded \kt module $$M=\bigoplus_{x \in H} M_x$$ such that each $M_x$ is endowed with a structure of an $A_x$\ti module (see Section~\ref{sect-modules-coalgebras}).

A \emph{$(\CM,A)$-linear morphism} between two $A$-modules  $M$ and $N$ is a \kt linear  homo\-mor\-phism $\alpha \co M \to N$ such that:
\begin{enumerate}
\labela
\item  For all $x\in H$,
$$\alpha(M_x) \subset \bigoplus_{e \in E} N_{\CM(e)x}.$$
\item For all $x \in H$ and $e \in E$, the \kt linear homomorphism $\alpha_{x,e} \co M_x \to N_{\CM(e)x}$ induced by $\alpha$ (restriction to $M_x$ followed by projection to $N_{\CM(e)x}$) is an $A_x$-linear morphism 
    $$
    \alpha_{x,e} \co M_x \to \phi_{x,e}^*(N_{\CM(e)x}),
    $$
    where  the $A_x$-module $\phi_{x,e}^*(N_{\CM(e)x})$ is the pullback of the $A_{\CM(e)x}$-mo\-du\-le $N_{\CM(e)x}$ along the algebra isomorphism $\phi_{x,e}\co A_x \to A_{\CM(e)x}$ given by the $\CM$\ti action of $A$. Graphically, the $A_x$-linearity of $\alpha_{x,e}$ depicts as:
    \begin{center}
\psfrag{x}[Br][Br]{\scalebox{.8}{$x$}}
\psfrag{y}[Br][Br]{\scalebox{.8}{$\CM(e)x$}}
\psfrag{e}[Br][Br]{\scalebox{.9}{$e$}}
\psfrag{M}[Bl][Bl]{\scalebox{.8}{$M_x$}}
\psfrag{N}[Bl][Bl]{\scalebox{.8}{$N_{\CM(e)x}$}}
\psfrag{f}[Bc][Bc]{\scalebox{.9}{$\alpha_{x,e}$}}
$\rsdraw{.45}{.9}{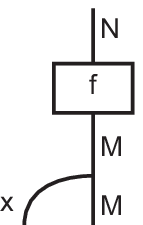}\quad=\quad\rsdraw{.45}{.9}{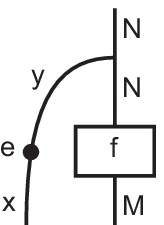}$\qquad .
\end{center}
\end{enumerate}
\medskip

\subsection{The categories of representations}\label{sect-modules-Xi-coalgebras}
We let $\Mod_\CM(A)$ be the category of $A$-modules and $(\CM,A)$-linear morphisms, with composition induced in the obvious way from the set-theoretical composition. The category $\Mod_\CM(A)$ is \kt linear (as a subcategory of $\Mod_\kk$). For any  two $A$-modules $M$ and $N$, we endow the \kt module $\Hom_{\Mod_\CM(A)}(M,N)$ with a structure of an $E$-graded \kt module by setting:
$$
\Hom_{\Mod_\CM(A)}(M,N)= \bigoplus_{e \in E} \Hom^e_{\Mod_\CM(A)}(M,N)
$$
where $\Hom^e_{\Mod_\CM(A)}(M,N)$ is the set of $(\CM,A)$-linear morphisms $\alpha \co M \to N$ such that for all $x \in E$, 
$$\alpha(M_x) \subset N_{\CM(e)x}.$$ 
The axioms \eqref{eq-Xi-act-1} and \eqref{eq-Xi-act-2} of a $\CM$\ti action imply that $\Mod_\CM(A)$ is $E$-Hom-graded. Also, since the $(\CM,A)$-linear morphisms of degree~1 are nothing but the $(H,A)$-linear morphisms, the 1-subcategory of $\Mod_\CM(A)$ is
$$
\Mod_\CM(A)^1=\Mod_H(A),
$$
where $\Mod_H(A)$ is the category of modules over the $H$-bicoal\-ge\-bra underlying~$A$ (see Section~\ref{sect-modules-coalgebras}). Note that for all $A$-modules $M,N$ and $e \in E$,
$$
\Hom^e_{\Mod_\CM(A)}(M,N)=\Hom_{\Mod_H(A)}\bigl(M,\phi^*_e(N)\bigr),
$$
where $\phi^*_e(N)$ is the $A$-module defined by 
$$
\phi^*_e(N)= \bigoplus_{x\in H} \phi^*_e(N)_x \quad \text{with} \quad \phi^*_e(N)_x=\phi_{x,e}^*(N_{\CM(e)x}) \in \Mod_\kk(A_x).
$$
The category $\Mod_\CM(A)$ is monoidal: the monoidal product of two $A$-modules,  the unit object,  and the monoidal constraints  of  $\Mod_\CM(A)$ are those of $\Mod_H(A)$, and the monoidal product of two $(\CM,A)$-linear morphisms is their tensor product over~$\kk$.

\begin{thm}\label{thm-mod-Xi-A}
The category $\Mod_\CM(A)$  is a $\CM$-graded monoidal category which is $E$-additive (see Section~\ref{sect-Hom-graded-categories}) and abelian. Moreover, if $A$ is a Hopf $\CM$-coalgebra, then $\Mod_\CM(A)$ is closed. 
\end{thm}

\begin{proof}
By the above, the category $\Mod_\CM(A)$ is $E$-Hom-graded and its 1-subcategory is $\Mod_\CM(A)^1=\Mod_H(A)$. In particular, $\Mod_\CM(A)^1$ is  $H$-graded (see Section~\ref{sect-modules-coalgebras}). The fact that the tensor product over $\kk$ of two $(\CM,A)$-linear morphisms is a $(\CM,A)$-linear morphism follows from \eqref{eq-Xi-act-3}. This implies that the above definitions do define a monoidal structure on $\Mod_\CM(A)$.

To prove that $\Mod_\CM(A)$ is $E$-additive, first remark that finite 1-direct sums of $A$-modules exist in $\Mod_\CM(A)$ (since $\Mod_H(A)$ is additive by Section~\ref{sect-modules-coalgebras}). Then, for any $e\in E$, the
$e$-direct sum of a finite family   $(M_\lambda)_{\lambda \in \Lambda}$ of $A$-modules exists in $\Mod_\CM(A)$ and is computed by
$$
\bigoplus^e_{\lambda\in \Lambda} M_\lambda= \bigoplus^1_{\lambda\in \Lambda} \phi^*_{e^{-1}}(M_\lambda).
$$

The fact that $\Mod_\CM(A)$ is  abelian is proved as in the classical case by noticing that the kernels and cokernels in $\Mod_\CM(A)$ are induced in the obvious way from those in $\Mod_\kk$. 

Finally, if $A$ is a Hopf $\CM$-coalgebra, then $\Mod_\CM(A)$ is closed since its 1-subcategory $\Mod_\CM(A)^1=\Mod_H(A)$ is closed (see Section~\ref{sect-modules-coalgebras}).
\end{proof}

Let $\mmod_\CM(A)$ be the full subcategory of  $\Mod_\CM(A)$ consisting of the $A$-modules whose underlying \kt module is projective of finite rank. By a \emph{pivotal element} of a  Hopf $\CM$-coalgebra, we mean a pivotal element of its underlying Hopf $H$-coalgebra (see Section~\ref{sect-modules-coalgebras}).

\begin{cor}\label{cor-mmod}
The category $\mmod_\CM(A)$ is an $E$-additive $\CM$-graded monoidal category.  Moreover, if $A$ is a Hopf $\CM$-coalgebra, then $\mmod_\CM(A)$ is rigid and the pivotal structures on $\mmod_\CM(A)$ are in bijective correspondence with the pivotal elements of~$A$.
\end{cor}
Note as above that the 1-subcategory of $\mmod_\CM(A)$ is
$$\mmod_\CM(A)^1=\mmod_H(A),$$
where $\mmod_H(A)$ is the $H$-graded monoidal category associated with the Hopf $H$\ti coalgebra underlying $A$ (see Section~\ref{sect-modules-coalgebras}).
\begin{proof}
The first assertion follows from the fact that the monoidal structure of  $\Mod_\CM(A)$ restricts to $\mmod_\CM(A)$.
Assume $A$ is a Hopf $\CM$-coalgebra. Then $\mmod_\CM(A)$ is rigid since its 1-subcategory $\mmod_\CM(A)^1=\mmod_H(A)$ is rigid (see Section~\ref{sect-modules-coalgebras}). Also, the pivotal structures on $\mmod_\CM(A)$ are in bijective correspondence with the pivotal structures on $\mmod_\CM(A)^1$, and so in bijective correspondence with the pivotal elements of $A$ (see Section~\ref{sect-modules-coalgebras}).
\end{proof}

\subsection{Remark}\label{rem-morph-Xi-coalgebras}
A homomorphism from a $\CM$-bicoalgebra $A$ to a $\CM$-bicoalgebra $B$ is an $H$-bicoalgebra homomorphism $f \co A \to B$ (see Remark~\ref{rem-morph-H-coalgebras})  which is compatible with the $\CM$-actions of $A$ and $B$ in the following sense: for all $(x,e) \in H \times E$,
$$
\phi^B_{x,e}f_x=f_{\CM(e)x} \phi^A_{x,e}.
$$
It not difficult to check that such a homomorphism $f\co A \to B$ induces $\CM$-graded functors $f^*\co \Mod_\CM(B) \to \Mod_\CM(A)$ and $f^*\co \mmod_\CM(B) \to \mmod_\CM(A)$. Here, by a  $\CM$\ti graded functor between $\CM$-graded monoidal categories, we mean a strong monoidal \kt linear functor that preserves the $H$-grading of objects and the $E$-grading of morphisms.

\subsection{Example}\label{ex-kk-CM}
Consider the trivial Hopf $\CM$-coalgebra $\kk_\CM= \{(\kk_\CM)_x=\kk\}_{x \in H}$ whose structural morphisms are given for all $x,y \in H$ and $e \in E$ by
$$
\Delta_{x,y}(1_{\kk})= 1_{\kk} \otimes 1_{\kk}, \quad \varepsilon= \id_{\kk}, \quad S_x = \id_{\kk}, \quad \phi_{x,e}=\id_{\kk}.
$$
By Theorem~\ref{thm-mod-Xi-A}, the category $\Mod_\CM(\kk_\CM)$ is an $E$-additive closed $\CM$-graded monoidal category. Its objects are the $H$-graded \kt modules $M=\bigoplus_{x \in H} M_x$ and its Hom-sets  are computed by
$$
\Hom^e_{\Mod_\CM(\kk_\CM)}(M,N)= \bigoplus_{x \in H} \Hom_\kk(M_x,N_{\CM(e)x}).
$$ 
The composition is induced from the set-theoretical composition. The monoidal product of objects is that of $H$-graded modules, the monoidal product of morphisms is their tensor product over~$\kk$, and the monoidal constraints are inherited from the category of \kt modules. Note that the 1-subcategory of $\Mod_\CM(\kk_\CM)$ is the category of $H$-graded \kt modules and  \kt linear  grading-preser\-ving homomorphisms (see Example~\ref{ex-kk-H}). Also, the $\CM$-category $\kk\tg_\CM$ from Example~\ref{ex-kk-GXi} is isomorphic to the
full subcategory of $\Mod_\CM(\kk_\CM)$ whose set of objects is $\{\kk_x\}_{x \in H}$, where $\kk_x$ is $\kk$ concentrated in degree $x$.

\subsection{Example}
Consider the Hopf $(E \to 1)$-coalgebra $\kk^\omega[G]$ from Example~\ref{ex-Hopf-xi-from-bicharacter}, where $E$ is an abelian group, $G$ is a group, and $\omega \co E \times G \to \kk^*$ is a bicharacter. By Theorem~\ref{thm-mod-Xi-A}, the category $\cc_\omega=\Mod_{(E \to 1)}(\kk^\omega[G])$ is an $E$-additive closed $(E \to 1)$-graded monoidal category. Its objects are the $\kk$-linear representations of $G$. For any representations $M,N$ of $G$ and any $e \in E$, the $\kk$-module $\Hom^e_{\cc_\omega}(M,N)$ is the set of \kt linear maps $\alpha \co M \to N$ such that for all $g \in G$ and $m \in M$,
$$
\alpha(g \cdot m)=\omega(e,g) \, \bigl(g \cdot \alpha(m)\bigr).
$$
The composition in $\cc_\omega$ is induced from the set-theoretical composition and the monoidal product in $\cc_\omega$ is the usual tensor product of representations of $G$. Note that these %the composition and monoidal product of morphisms
are $E$-multiplicative as expected (see Section~\ref{sect-part-cases}).

\subsection{Example}
Consider the Hopf $\CM$-coalgebra $A_\CM^\rho$ from Example~\ref{ex-Hopf-xi-from-action}, where $A$ is a Hopf algebra and  $\rho \co H \to \Aut_{\alg}(A)$ is a group homomorphism.
By Theorem~\ref{thm-mod-Xi-A}, the category $\Mod_\CM(A_\CM^\rho)$ is an $E$-additive closed $\CM$-graded monoidal category. Its objects are the $H$-graded $A$-modules, that is, the $A$-modules $M$ endowed with an $H$-grading $M=\bigoplus_{x \in H} M_x$ so that $a \cdot M_x \subset M_x$ for all $a \in A$ and $x \in H$.
The Hom-sets  are given by
$$
\Hom^e_{\Mod_\CM(\kk_\CM)}(M,N)= \bigoplus_{x \in H} \Hom_{\Mod_\kk(A)}\bigl(M_x,\rho_{\CM(e)}^*(N_{\CM(e)x})\bigr).
$$
The composition is induced from the set-theoretical composition. The monoidal product is induced from the usual monoidal product of $A$-modules (using the coproduct of $A$) and of $H$-graded \kt modules.
 Note that when $A=\kk$ we recover Example~\ref{ex-kk-CM} (since $\rho$ becomes trivial and  $\kk_\CM^\rho=\kk_\CM$ as Hopf $\CM$-coalgebras).

\subsection{$\CM$-fusion categories from Hopf $\CM$-coalgebras}\label{sect-semisimple}
We first recall the definition of a $\CM$-fusion category from \cite{SV}. By a  \emph{simple} object of a $\kk$-linear category, we mean an object whose set of endomorphisms is a free \kt module of rank 1. %It is clear that an object isomorphic to a  simple object is itself  simple.
Note that if $\cc$ is a $\CM$-category over~$\kk$, then a simple object of  the 1-subcategory $\cc^1$ of~$\cc$ is nothing but an object $i$ of $\cc$ such that $\End_\cc^1(i)=\kk \, \id_i$.

A \emph{$\CM$-fusion category (over $\kk$)} is a rigid $\CM$-category $\cc$ (over $\kk$) such that:
\begin{enumerate}
\labela
\item The 1-subcategory $\cc^1$ of $\cc$ is \emph{$H$-fusion}, that is:
\begin{itemize} 
\item $\Hom_{\cc^1}(i,j)=0$ whenever $i,j$ are non-isomorphic simple objects of $\cc^1$,
\item each object of $\cc^1$ is a (finite) direct sum of simple objects of~$\cc^1$,
\item the unit object $\un$ is a simple object of $\cc^1$,
\item for any $h \in H$, there are at least one and only finitely many (up to isomorphism in $\cc^1$) homogeneous simple objects of $\cc^1$ with degree $h$.
\end{itemize}
\item For any $e \in E$, each object of $\cc$ is an $e$-direct sum of a finite family of simple objects of $\cc^1$.
\end{enumerate} 
For example, the $\CM$-category $\kk\tg_\CM$ from Example~\ref{ex-kk-GXi} is $\CM$-fusion. 

Clearly, the Hom-sets in a $\CM$-fusion category are free $\kk$-modules of finite rank.
Note that a $\CM$-fusion category $\cc$ may not be semisimple (as a \kt linear category) while its $1$-subcategory $\cc^1$ always is. Indeed, a simple object $i$ of $\cc^1$ is not necessarily simple in $\cc$: it may happen that $\End^e_\cc(i)\neq 0$ for some $e \in E \setminus\{1\}$ (as for example in $\kk\tg_\CM$ when $\Ker(\CM)$ is nontrivial).

\begin{thm}\label{thm-Xi-fusion}
Let $A=\{A_x\}_{x \in H}$ be a Hopf $\CM$-coalgebra over an algebraically closed field $\kk$ such that the \kt algebra $A_1$ is semisimple and  for all $x \in H$, the \kt algebra $A_x$ is nonzero and finite dimensional. Then the category $\mmod_\CM(A)$ from Corollary~\ref{cor-mmod} is $\CM$-fusion.
\end{thm}

\begin{proof}
By \cite[Lemma 5.1]{Vi2} applied to the Hopf $H$-coalgebra underlying $A$, we obtain that for all $x \in H$, the \kt algebra $A_x$ is semisimple, and so the category of finite dimensional  $A_x$-modules is semisimple.  Note that all the irreducible $A_x$-modules are simple objects in the above sense because $\kk$ is an algebraically closed field, and there are at least one and only finitely many of them (up to isomorphism) since $A_x$ is nonzero and finite dimensional. Also the unit object (which is $\kk$ with $A_1$-action given by the counit) is an irreducible $A_x$-module. Then $\mmod_\CM(A)^1=\mmod_H(A)$ is $H$-fusion. 

It remains to verify Axiom $(b)$. Let $e \in E$ and $X$ be an object of $\mmod_\CM(A)$. Since $\mmod_\CM(A)^1$ is $H$-fusion, the object $X$ is a $1$-direct sum of a finite family $(i_\lambda)_{\lambda \in \Lambda}$ of homogeneous simple objects of $\mmod_\CM(A)^1$. It follows from the definition of $\mmod_\CM(A)$ that if $i$ is a homogeneous simple object of $\mmod_\CM(A)^1$, then $(\phi_{\CM(e^{-1})|i|,e})^*(i)$ is a homogeneous simple object of $\mmod_\CM(A)^1$ which is $e$-isomorphic to $i$.
Consequently, the object $X$ is the $e$-direct sum of the finite family $\bigl((\phi_{\CM(e^{-1})|i_\lambda|,e})^*(i_\lambda)\bigr)_{\lambda \in \Lambda}$ of homo\-ge\-neous simple objects of $\mmod_\CM(A)^1$.
\end{proof}

\subsection{Representations of Hopf $\CM$-algebras}\label{sect-dual-A-comodules}
Recall from Section~\ref{sect-dual-notion} the notions of a $\CM$-bialgebra and of a Hopf $\CM$-algebra.
Any $\CM$-bialgebra $A=\bigoplus_{x \in H} A_x$ (over~$\kk$) yields an $E$-additive abelian $\CM$-graded monoidal category $\Comod_\CM(A)$. The construction of~$\Comod_\CM(A)$ is dual to that given in Section~\ref{sect-modules-Xi-coalgebras} and we only give here a brief description of it.

The objects of  $\Comod_\CM(A)$ are the (left) $A$-comodules, that is, the $H$\ti graded \kt module $M=\bigoplus_{x \in H} M_x$ such that each $M_x$ is endowed with a structure of a (left) $A_x$\ti comodule. An $A$-comodule $M$ is homogeneous of degree $x \in H$ if it is nonzero and concentrated in degree $x$ (that is, $M=M_x$). The monoidal product of $A$-comodules is induced by the usual tensor product of $H$-graded \kt modules and the product $\mu=\{\mu_{x,y} \co A_x \otimes A_y \to A_{xy}\}_{x,y \in H}$ of $A$. 

For all $M,N\in \Comod_\CM(A)$ and $e \in E$, the \kt module $\Hom^e_{\Comod_\CM(A)}(M,N)$ is the set of  \kt linear  homo\-mor\-phisms $\alpha \co M \to N$ such that for all $x\in H$,
\begin{enumerate}
\labela
\item  $\alpha(M_x) \subset N_{\CM(e)x}$,
\item the map $(\phi_{x,e})_*(M_x) \to N_{\CM(e)x}$  induced by $\alpha$ is $A_{\CM(e)x}$-colinear, where $\phi$ is the $\CM$-action of $A$.
\end{enumerate}
The composition is induced in the obvious way from the set-theoretical composition. The monoidal product of morphisms is their tensor product over $\kk$. The  monoidal constraints  of  $\Comod_\CM(A)$ are induced (in the obvious way) from those of $\Mod_\kk$.

Similarly, the full subcategory $\comod_\CM(A)$ of  $\Comod_\CM(A)$ consisting of the $A$\ti co\-mo\-du\-les whose underlying \kt module is projective of finite rank is an $E$-additive $\CM$-graded monoidal category. 

Dually to Theorem~\ref{thm-mod-Xi-A} and Corollary~\ref{cor-mmod}, if $A$ is a Hopf $\CM$-algebra, then the $\CM$-category $\Comod_\CM(A)$ is closed and the $\CM$-category $\comod_\CM(A)$ is rigid. Also, the dual $A^*=\{A_x^*\}_{x \in H}$ of a Hopf $\CM$-algebra $A$ of finite type is a Hopf $\CM$-coalgebra (see Section~\ref{sect-dual-notion}) and the $\CM$-graded monoidal categories $\mmod_\CM(A^*)$ and $\comod_\CM(A)$ are isomorphic.
Conversely, the dual $A^*=\bigoplus_{x \in H} A_x^*$ of a Hopf $\CM$-coalgebra $A=\{A_x\}_{x \in H}$ of finite type is a Hopf $\CM$-algebra and the $\CM$-graded monoidal categories $\comod_\CM(A^*)$ and $\mmod_\CM(A)$ are  isomorphic.

\section{Hopf crossed module-modules and integrals}\label{sect-Hopf-modules-integrals} 
Throughout this section, $\CM \co E \to H$ is a crossed module and $A=\{A_x\}_{x \in H}$ is a Hopf $\CM$-coalgebra (over $\kk$). We introduce Hopf $\CM$-modules over $A$ and prove a structure theorem for them. Next we use this theorem to prove the existence and uniqueness of $\CM$-integrals for $A$.

\subsection{Hopf $\CM$-modules}\label{sect-Hopf-modules-Xi-coalgebras}
A \emph{(left-left) Hopf $\CM$-module} over the Hopf $\CM$-coalgebra~$A$ is a  family $M=\{M_x\}_{x \in H}$ of \kt modules endowed with three families
\begin{align*}
& r=\{r_x \co A_x \otimes M_x \to M_x\}_{x \in H}, \\
& \rho=\{\rho_{x,y} \co M_{xy} \to A_x \otimes M_y\}_{x,y \in H}, \\
&\psi=\{\psi_{x,e} \co M_x \to M_{\CM(e)x} \}_{(x,e) \in H \times E},
\end{align*}
of \kt linear homomorphisms such that:
\begin{enumerate}
  \labela
  \item  For all $x \in H$, the pair $(M_x,r_x)$ is a (left) $A_x$-module.
  \item  The pair $(M,\rho)$ is a (left) $A$-comodule, that is, for all $x,y,z \in H$,
$$
(\Delta_{x,y}\otimes \id_{M_z}) \rho_{xy,z}=(\id_{A_x} \otimes \rho_{y,z}) \rho_{x,yz} \quad \text{and} \quad
(\varepsilon \otimes \id_{M_x}) \rho_{1,x}=\id_{M_x}.
$$
  \item The action $r$ and coaction $\rho$ intertwine as follows: for all $x,y \in H$,
$$
\rho_{x,y}r_{xy}=(\mu_x \otimes r_y)(\id_{A_x} \otimes \sigma_{A_y,A_x} \otimes \id_{M_y})(\Delta_{x,y} \otimes \rho_{x,y}).
$$
\item For all  $x,y \in H$ and $e,f \in E$,
\begin{align*}
 \psi_{x,1}&=\id_{M_x}, &\psi_{e,x}r_x&=r_{\CM(e)x} (\phi_{x,e} \otimes \psi_{x,e}), \\
\psi_{x,fe} &=\psi_{\CM(e)x,f}\, \psi_{x,e}, & (\phi_{x,e} \otimes \psi_{y,f}) \rho_{x,y}&=\rho_{\CM(e)x,\CM(f)y}\, \psi_{xy,e\lact{x}{\!f}}.
\end{align*}
\end{enumerate}
Here $\mu=\{\mu_x\}_{x \in H}$, $\Delta=\{\Delta_{x,y}\}_{x,y \in H}$, $\varepsilon$, and  $\phi=\{\phi_{x,e}\}_{(x,e) \in H \times E}$ are the product, coproduct, counit, and $\CM$-action of $A$, respectively. Note that $\bigoplus_{x \in H} M_x$ is then  an $A$-module in the sense of Sections~\ref{sect-modules-coalgebras} and~\ref{sect-modules-Xi-coalgebras-def}.

\subsection{Morphisms of Hopf $\CM$-modules}\label{sect-morphisms-Hopf-modules}
Consider two Hopf $\CM$-modules $M=(M,r,\rho,\psi)$ and  $M'=(M',r',\rho',\psi')$ over $A$. A \emph{morphism of Hopf $\CM$-modules} from~$M$ to~$M'$ is a family $\theta=\{\theta_x\co M_x \to M'_x\}_{x \in H}$ of \kt linear homomorphisms such that
each $\theta_x$ is $A_x$-linear,  $\theta$ is a morphism of $A$-comodules, and $\theta$ is a $\CM$\ti equivariant:
$$
r'_x(\id_{A_x} \otimes  \theta_x) = \theta_x r_x,\quad
(\id_{A_x} \otimes \theta_y) \rho_{x,y}= \rho'_{x,y} \theta_{xy}, \quad \theta_{\CM(e)x}\psi_{x,e}=\psi'_{x,e} \theta_x
$$
for all $x,y \in H$ and $e \in E$.

Clearly the composition (componentwise) of two morphisms of Hopf $\CM$-modules is a morphism of Hopf $\CM$-modules.

\subsection{Modules of coinvariants}\label{sect-coinvariants}
The \emph{module of coinvariants} of a  Hopf $\CM$-module $M=(M,r,\rho,\psi)$ over $A$  is the \kt submodule $M^{\mathrm{co} A}$ of $\prod_{x \in H} M_x$ consisting of the elements $m=(m_x)_{x \in H}$ such that for all $x,y \in H$ and $e \in E$,
$$
\rho_{x,y} (m_{xy})=1_x \otimes m_y \quad \text{and} \quad  \psi_{x,e}  (m_x)=m_{\CM(e)x},
$$
where $1_x$ is the unit element of $A_x$.

Any morphism of Hopf $\CM$-modules $\theta \co M \to M'$ induces a \kt linear homomorphism
$\theta^{\mathrm{co} A} \co M^{\mathrm{co} A}  \to (M')^{\mathrm{co} A}$ defined by
$$
\theta^{\mathrm{co} A}\bigl((m_x)_{x \in H} \bigr)=\bigl(\theta_x(m_x)\bigr)_{x \in H}.
$$

\subsection{Trivial Hopf $\CM$-modules}\label{sect-exa-trivial-Hopf-module}
Let $V$ be a \kt module. Then
$$
\left ( \{A_x \otimes V\}_{x \in H}, \{\mu_x \otimes \id_V\}_{x \in H}, \{\Delta_{x,y}\otimes \id_V\}_{x,y \in H},
\{\phi_{x,e} \otimes \id_V \}_{(x,e) \in H \times E}\right )
$$
is a Hopf $\CM$-module over $A$, denoted $A \otimes V$. Its module of coinvariants is
$$
(A \otimes V)^{\mathrm{co} A}=\{(1_x \otimes v)_{x \in H} \, | \, v \in V \}.
$$
Note that the assignment $v \mapsto (1_x \otimes v)_{x \in H}$ is a \kt linear isomorphism $V\cong(A \otimes V)^{\mathrm{co} A}$  with inverse $(m_x)_{x \in H} \mapsto (\varepsilon \otimes \id_V)(m_1)$.

Clearly, if $\alpha \co V \to W$ is a \kt linear homomorphism, then $\{\id_{A_x} \otimes \alpha \}_{x \in H}$ is a morphism of Hopf $\CM$-modules from $A \otimes V$ to $A \otimes W$.

\subsection{Structure of Hopf $\CM$-modules}\label{sect-Hopf-modules-structure}
Let $\Mod_\kk$ be the category of \kt modules and \kt linear homomorphisms. Denote by $\Hmod{A}$ the category of Hopf $\CM$-modules over $A$ and their morphisms. The trivial Hopf $\CM$-modules (see Section~\ref{sect-exa-trivial-Hopf-module}) and the modules of coinvariants (see Section~\ref{sect-coinvariants}) induce the following functors: 
$$
A \otimes ? \co  \Mod_\kk \to \Hmod{A} \quad \text {and} \quad ?^{\mathrm{co} A} \co \Hmod{A} \to \Mod_\kk.
$$

\begin{thm}\label{thm-structure-Hopf-Xi-modules}
The functors $A \otimes ?$ and $?^{\mathrm{co} A}$ are equivalences and are quasi-inverse of each other. In particular, any  Hopf $\CM$-module $M$ over $A$ is isomorphic to the trivial Hopf $\CM$-module $A \otimes M^{\mathrm{co} A}$.
\end{thm}

\begin{proof}
It follows from the definitions that $?^{\mathrm{co} A}$ is right adjoint to  $A \otimes ?$ with unit $\eta_V \co V \to (A \otimes V)^{\mathrm{co} A}$ and counit $\epsilon_M=\{\epsilon_M^x \co A_x \otimes M^{\mathrm{co} A} \to M_x\}_{x \in H}$ given by
$$
\eta_V(v)=(1_x \otimes v)_{x \in H} \quad \text {and} \quad \epsilon_M^x\bigl(a \otimes (m_y)_{y \in H}\bigr)=r_x(a \otimes m_x).
$$
By Section~\ref{sect-exa-trivial-Hopf-module}, the unit $\eta$ is an isomorphism. Let us prove that the counit $\epsilon$ is also an isomorphism. Let $M$ be any Hopf $\CM$-module over $A$. To prove that $\epsilon_M$ is an isomorphism of Hopf $\CM$-modules from  $A \otimes M^{\mathrm{co} A}$ to $M$, we exhibit its inverse: defining $\pi \co M_1 \to M^{\mathrm{co} A}$ as
$$
\pi(m)=\left (r_x(S_x \otimes \id_{M_x})\rho_{x^{-1},x}(m)\right)_{x \in H},
$$
where $S=\{S_x\}_{x \in H}$ is the antipode of $A$, it follows from the definitions that
$$
\nu_M=\{\nu_M^x=(\id_{A_x} \otimes \pi)\rho_{x,1} \co M_x \to A_x \otimes M^{\mathrm{co} A} \}_{x \in H}
$$
is the inverse of $\epsilon_M$.
\end{proof}

\subsection{Integrals}\label{sect-integ-def}
A \emph{left} (respectively, \emph{right}) \emph{$\CM$-integral} for $A$ is a family of $\kk$-linear forms $\lambda=(\lambda_x \co A_x \to \kk)_{x \in H}$ such that for all $x,y \in H$ and $e \in E$,
\begin{itemize}
\item $(\id_{A_x} \otimes \lambda_y) \Delta_{x,y} = \eta_x\lambda_{xy}$  \quad (respectively, $(\lambda_x \otimes \id_{A_y}) \Delta_{x,y} = \eta_y\lambda_{xy}$),
\item $\lambda_{\CM(e)x}\phi_{x,e} =\lambda_x$.
\end{itemize}
Here $\{\eta_x \co \kk \to A_x\}_{x \in H}$,  $\Delta=\{\Delta_{x,y}\}_{x,y \in H}$, and $\phi=\{\phi_{x,e}\}_{(x,e) \in H \times E}$ are the unit maps, coproduct, and $\CM$-action of $A$.

We denote by $\mathcal{I}_A^l$ (respectively, $\mathcal{I}_A^r$) the set of left (respectively, right) $\CM$-integrals for $A$.
The sets  $\mathcal{I}_A^l$ and $\mathcal{I}_A^r$ are \kt modules (as submodules of the \kt module $\prod_{x \in H} A_x^*$).  They are isomorphic: it follows from the properties of the antipode $S$ of~$A$ (see Section~\ref{sect-Hopf-H-coalgebras-antipodes} and Lemma~\ref{lem-Xi-antipode}) that the map
$$
\mathcal{I}_A^l \to \mathcal{I}_A^r, \quad \lambda=(\lambda_x)_{x \in H}  \mapsto \lambda^S=(\lambda_x^S=\lambda_{x^{-1}}S_{x^{-1}})_{x \in H}
$$
is a \kt linear isomorphism.

Note that a $\CM$-integral for $A$ is in particular an $H$-integral (in the sense of \cite{Vi2}) for the Hopf $H$-coalgebra underlying $A$.  (An $H$-integral for a Hopf $H$-coalgebra verifies only the first of the above two axioms of a $\CM$-integral.)
Consequently, by \cite[Lemma 3.1]{Vi2}, if a left or right $\CM$-integral $\lambda=(\lambda_x)_{x \in H}$ for $A$ is non-zero, then $\lambda_x\neq 0$ for all $x \in H$ such that $A_x \neq 0$, and in particular $\lambda_1 \neq 0$.

\subsection{Existence and uniqueness of $\CM$-integrals}
It is known (see \cite{sweed2}) that the space of left (respectively, right) integrals for a finite
dimensional Hopf algebra over a field is one dimensional. We generalize this result to Hopf $\CM$-coalgebras:

\begin{thm}\label{thm-existence_integrals}
Assume that $\kk$ is a field and that $A$ is a Hopf $\CM$-coalgebra (over $\kk$) of finite type (that is, each~$A_x$ is finite dimensional).
Then the spaces $\mathcal{I}_A^l$ and $\mathcal{I}_A^r$ are both one dimensional. 
\end{thm}

\begin{proof}
For any $x,y \in H$ and $e \in E$, set $M_x=A_{x^{-1}}^*=\Hom_\kk(A_{x^{-1}},\kk)$  and define the \kt linear homomorphisms
$$
r_x\co A_x \otimes M_x \to M_x, \quad \rho_{x,y} \co M_{xy} \to A_x \otimes M_y, \quad \psi_{x,e}\co M_x \to M_{\CM(e)x}
$$
by
$$
\psfrag{x}[Br][Br]{\scalebox{.8}{$x$}}
\psfrag{a}[Bl][Bl]{\scalebox{.8}{$x^{-1}$}}
\psfrag{u}[Bl][Bl]{\scalebox{.9}{$y^{-1}$}}
\psfrag{c}[Bl][Bl]{\scalebox{.9}{$(xy)^{-1}$}}
\psfrag{n}[Bl][Bl]{\scalebox{.9}{$(\CM(e)x)^{-1}$}}
\psfrag{e}[Bl][Bl]{\scalebox{.9}{$\lact{{x^{-1}}}{\!\!e}$}}
r_x=\rsdraw{.45}{.9}{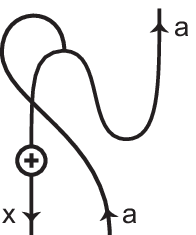} \qquad \rho_{x,y}=\rsdraw{.45}{.9}{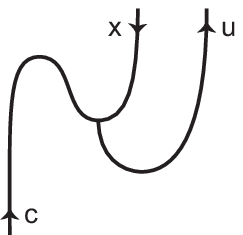}  \qquad \psi_{x,e}=\rsdraw{.45}{.9}{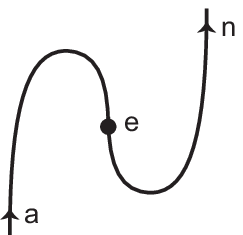} \quad.
$$
Here we use the convention of Section~\ref{sect-Penrose} for oriented arcs, a color $z \in H$ is an abbreviation for $A_z$ (as in Section~\ref{sect-Hopf-H-coalgebras}), and the dot represents the $\CM$-action of $A$ (see Section~\ref{sect-crossed-module-actions}).
It follows from the axioms of a Hopf $\CM$-coalgebra and Lemma~\ref{lem-Xi-antipode} that
$$
M=\bigl(M=\{M_x\}_{x \in H}, r=\{r_x\}_{x \in H}, \rho=\{\rho_{x,y}\}_{x,y \in H},\psi=\{\psi_{x,e}\}_{(x,e) \in H \times E}\bigr)
$$
is a Hopf $\CM$-module over $A$. Then, by Theorem~\ref{thm-structure-Hopf-Xi-modules}, $M$ is isomorphic to the trivial Hopf $\CM$-module $A \otimes M^{\mathrm{co} A}$. Now it follows from the definitions of $\rho$, $\psi$, and of a right $\CM$-integral that the map
$$
\mathcal{I}_A^r \to M^{\mathrm{co} A}, \quad \lambda=(\lambda_x)_{x \in H}  \mapsto \lambda^{\mathrm{co}}=(\lambda_{x^{-1}})_{x \in H}
$$
is a \kt linear isomorphism. Thus $M$ is isomorphic to the trivial Hopf $\CM$-module $A \otimes \mathcal{I}_A^r$. In particular the \kt vector spaces $M_1$ and $A_1 \otimes\mathcal{I}_A^r$ are isomorphic and so
$$
\dim(M_1)=\dim(A_1 \otimes\mathcal{I}_A^r)=\dim(A_1)\dim(\mathcal{I}_A^r).
$$
Now $\dim(M_1)=\dim(A_1)$ (since $M_1=A_1^*$ and $A_1$ is finite dimensional) and $\dim(A_1)\neq0$  (since $A_1 \neq 0$ because $\varepsilon(1_1)=1_\kk \neq 0$). Consequently $\dim(\mathcal{I}_A^r)=1$. Using that $\mathcal{I}_A^l$ is isomorphic to $\mathcal{I}_A^r$ (see Section~\ref{sect-integ-def}), we get that $\dim(\mathcal{I}_A^l)=1$.
\end{proof}

\subsection{The distinguished $\CM$-grouplike element}\label{sesct-disting-grouplike}
As in Theorem~\ref{thm-existence_integrals}, we assume in this subsection that $\kk$ is a field and that $A$  is  a Hopf $\CM$-coalgebra (over $\kk$) of finite type.

\begin{cor}\label{cor-disting-grouplike}
There exists a unique $\CM$-grouplike element $g=(g_x)_{x \in H}$ of $A$ such that $
 (\id_{A_x} \otimes \lambda_y)\Delta_{x,y}=  g_x \lambda_{xy}$ for any right $\CM$-integral $\lambda$ of $A$ and all $x,y \in H$.
\end{cor}
\begin{proof}
First notice that the space $\mathcal{I}_A^r$ of right $\CM$-integrals for $A$ coincide with the space of right $H$-integrals for the Hopf $H$-coalgebra underlying $A$ because both are one dimensional (by Theorem~\ref{thm-existence_integrals} and \cite[Theorem 3.6]{Vi2}) and any right $\CM$-integral is a right $H$-integral. Then, by \cite[Lemma 4.1]{Vi2}, there is a unique $H$-grouplike element $g=(g_x)_{x \in H}$ of $A$ satisfying the condition of the corollary. Now it follows from the properties of the $\CM$-action $\phi$ of $A$ that $\bigl(\phi_{\CM(e)x,e^{-1}}(g_{\CM(e)x})\bigr)_{x \in H}$ is also an $H$-grouplike element of $A$ satisfying the condition of the lemma. Then the uniqueness of such an $H$-grouplike element implies that $\phi_{\CM(e)x,e^{-1}}(g_{\CM(e)x})=g_x$, and so $g_{\CM(e)x}=\phi_{x,e}(g_x)$, for all $x \in H$ and $e \in E$. Hence $g$ is a $\CM$-grouplike element.
\end{proof}

The $\CM$-grouplike element $g=(g_x)_{x \in H}$ of Corollary~\ref{cor-disting-grouplike} is called the \emph{distinguished
$\CM$-grouplike element} of $A$. Note that $g_1$ is the (usual) distinguished grouplike element of the Hopf algebra
$A_1$.  Also $g$ is the distinguished $H$-grouplike element of the Hopf $H$-coalgebra underlying $A$ (see \cite[Section 4]{Vi2}).

\section{Hopf crossed module-(co)algebras as Hopf algebras}\label{sect-Hopf-Xi-as-Hopf}

Throughout this section, we let $\sss=(\sss,\otimes,\un)$  be a symmetric monoidal category. We first define Hopf crossed module-(co)algebras in $\sss$ and then interpret them as Hopf algebras in some symmetric monoidal category associated with $\sss$. As a corollary, when $\sss$ is the category of \kt modules, we obtain that  the Hopf crossed module-(co)algebras over $\kk$ considered above are Hopf algebras in some symmetric monoidal category.

\subsection{Hopf crossed module-(co)algebra in symmetric monoidal categories}
Let $\CM \co E \to H$ be a crossed module. The notion of a Hopf $\CM$-(co)algebra in $\sss$ can be defined in the exact same way as in Sections~\ref{sect-Hopf-crossed-module-coalgebras} and \ref{sect-dual-notion} by replacing \kt modules with objects of $\sss$ and \kt linear homomorphisms with morphisms in $\sss$.

Explicitly, a \emph{Hopf $\CM$-coalgebra in $\sss$} is a family $A=\{A_x\}_{x \in H}$ of algebras in~$\sss$ endowed with a coproduct, a counit, an antipode, and a $\CM$-action which are respectively:
\begin{itemize}
\item a family $\Delta=\{\Delta_{x,y}\co A_{xy} \to A_x \otimes A_y \}_{x,y \in H}$  of algebra morphisms in $\sss$, 
\item an algebra morphism $\varepsilon\co A_1 \to \un$ in $\sss$,
\item a family $S=\{S_x \co A_{x^{-1}} \to A_x\}_{x \in H}$  of isomorphisms in $\sss$, 
\item a family $\phi=\{\phi_{x,e} \co A_x \to A_{\CM(e)x} \}_{(x,e) \in H \times E}$ of algebra isomorphisms in $\sss$.
\end{itemize}
These data should satisfy the following axioms: the coproduct must be coassociative and  counital 
as in Section~\ref{sect-label-H-coalgebras}, the antipode must satisfy the axiom of Section~\ref{sect-Hopf-H-coalgebras-antipodes}, and the $\CM$-action must satisfy the axioms \eqref{eq-Xi-act-1}-\eqref{eq-Xi-act-3} of Section~\ref{sect-crossed-module-actions}.

Similarly, a \emph{Hopf $\CM$-algebra in $\sss$} is a family $A=\{A_x\}_{x \in H}$ of coalgebras in~$\sss$ endowed with a product, a unit, an antipode, and a $\CM$-action which are respectively:
\begin{itemize}
\item a family $\mu=\{\mu_{x,y}\co A_x \otimes A_y \to A_{xy} \}_{x,y \in H}$  of coalgebra morphisms in $\sss$, 
\item a coalgebra morphism $\eta\co \un \to A_1$ in $\sss$,
\item a family $S=\{S_x \co A_x \to A_{x^{-1}}\}_{x \in H}$  of isomorphisms in $\sss$, 
\item a family $\phi=\{\phi_{x,e} \co A_x \to A_{\CM(e)x} \}_{(x,e) \in H \times E}$ of coalgebra isomorphisms in~$\sss$,
\end{itemize}
and which satisfy the axioms given in Section~\ref{sect-dual-notion}.

Note that these notions are dual to each other: Hopf $\CM$-algebras in $\sss$ bijectively correspond to Hopf $\CM$-coalgebras in the opposite category $\sss^\opp=(\sss^\opp,\otimes, \un)$ by exchanging the product and the coproduct and by replacing the  $\CM$-action $\phi$ with $\phi^{-1}=\{\phi^{-1}_{x,e} =\phi_{\CM(e)x,e^{-1}}\}_{(x,e) \in H \times E}$.

For example, let $\Mod_\kk$ be the symmetric monoidal category of  \kt modules and \kt linear homomorphisms. Then the Hopf $\CM$-coalgebras over $\kk$ (as defined in Section~\ref{sect-Hopf-crossed-module-coalgebras}) are exactly the Hopf $\CM$-coalgebras in $\Mod_\kk$. Likewise, the Hopf $\CM$-algebras over $\kk$ (as defined in Section~\ref{sect-dual-notion}) are exactly the Hopf $\CM$-algebras in $\Mod_\kk$ and also correspond to the Hopf $\CM$-coalgebras in $(\Mod_\kk)^\opp$.

\subsection{The category $\ee(\sss)$}\label{sect-ES}
We associate to $\sss$ a symmetric monoidal category $\ee(\sss)$ defined as follows.

The objects of $\ee(\sss)$ are the pairs $(\cc,F)$ where $\cc$ is a small category and $F\co \cc \to \sss$ is a functor. A morphism from an object $(\cc,F)$ to an object  $(\dd,G)$ is a pair $(\Gamma,\gamma)$ where $\Gamma\co \dd \to \cc$ is a functor and $\gamma=\{\gamma_Y\co F\Gamma(Y) \to G(Y)\}_{Y \in \dd}$ is a natural transformation (from $F \Gamma$ to $G$). The composition of $(\Gamma,\gamma)\co (\cc,F) \to (\dd,G)$ with $(\Lambda,\lambda)\co (\dd,G) \to (\mathcal{K},H)$ is defined by
$$
(\Lambda,\lambda) \circ (\Gamma,\gamma)=(\Gamma \Lambda, \lambda \gamma_\Lambda) \quad \text{where} \quad
\lambda \gamma_\Lambda=\{\lambda_Z \circ \gamma_{\Lambda(Z)}\}_{Z\in \mathcal{K}}.
$$
The identity of $(\cc,F)$ is $(1_\cc,\id_F)$ where $1_\cc \co \cc \to \cc$ is the identity functor and $\id_F=\{\id_{F(X)}\}_{X \in \cc}$. 
The monoidal product of two objects $(\cc,F)$ and $(\dd,G)$ is defined by
$$
(\cc,F) \boxtimes (\dd,G)=\bigl(\cc \times \dd, F\otimes G=\otimes \circ (F \times G) \co \cc \times  \dd \to \sss\bigr).
$$
The monoidal product of a morphism $(\Gamma,\gamma)\co (\cc,F) \to (\cc',F')$ with a morphism $(\Lambda,\lambda)\co (\dd,G) \to (\dd',G')$ is defined by
$$
(\Gamma,\gamma) \boxtimes (\Lambda,\lambda)=\bigl(\Gamma \times \Lambda,\gamma \otimes \lambda=\{\gamma_X \otimes \lambda_Y\}_{(X,Y) \in \cc' \times \dd'} \bigr).
$$
This yields a functor $\boxtimes \co \ee(\sss) \times \ee(\sss) \to \ee(\sss)$.
Set $I=(\mathbf{1},\underline{\un})$, where  $\mathbf{1}$  is the trivial category with a single object $\ast$ and a single morphism $\id_\ast$, and the functor $\underline{\un} \co \mathbf{1} \to \sss$ is defined by $\underline{\un}(\ast)=\un$. Then $\ee(\sss)=(\ee(\sss), \boxtimes,I)$ is a monoidal category (with the obvious monoidal constraints) and is symmetric with symmetry induced (in the obvious way) from the symmetries of $\Cat$ (see Section~\ref{sect-cartesian-monoidal-categories}) and $\sss$.

Note that the forgetful functor $\mathcal{U}\co\ee(\sss) \to \Cat^\opp$, defined by $\mathcal{U}(\cc,F)=\cc$ and $\mathcal{U}(\Gamma,\gamma)=\Gamma$,
is strict monoidal and symmetric.

\subsection{Hopf algebras in $\ee(\sss)$}\label{sect-Hopf-alg-ES}
In the next theorem, we characterize Hopf crossed module-coalgebras in $\sss$ as Hopf algebras in $\ee(\sss)$.
\begin{thm}\label{thm-characterization-Hopf-Xi-coalgebras-in-S}
Hopf algebras in $\ee(\sss)$ are in bijective correspondence with pairs $(\CM,A)$ where $\CM$ is a crossed module and $A$ is a Hopf $\CM$-coalgebra in $\sss$.
\end{thm}
\begin{proof}
Let $(\cc,F)$ be a Hopf algebra in the category $\ee(\sss)$. Then $\cc=\mathcal{U}(\cc,F)$ is a Hopf algebra in $\Cat^\opp$. By Section~\ref{sect-crossed-modules-as-Hopf-algebras}  and since Hopf algebras in $\Cat^\opp$ coincide with Hopf algebras in $\Cat$ (by exchanging product and coproduct), there is a unique crossed module $\CM \co E \to H$ such that
$\cc=(\tg_\CM,\Delta_\CM,\varepsilon_\CM,m_\CM,u_\CM,S_\CM)$ as Hopf algebras in $\Cat^\opp$. The functor $F\co \tg_\CM \to \sss$ gives rise to the family $A=\{A_x=F(x)\}_{x \in H}$ of objects in $\sss$ and to the family
$$
\phi=\left\{\phi_{x,e}=F\bigl(x \xrightarrow{e} \CM(e)x\bigr) \co A_x \to A_{\CM(e)x} \right\}_{(x,e) \in H \times E}
$$
of morphisms in $\sss$. The functoriality of $F$ amounts to Axioms~\eqref{eq-Xi-act-1} and~\eqref{eq-Xi-act-2}.
The product of the Hopf algebra $(\cc,F)$ is $(\Delta_\CM,\mu)$ where $\mu=\{\mu_x \co A_x \otimes A_x \to A_x\}_{x \in H}$ is a natural transformation from $(F \otimes F)\Delta_\CM$ to $F$. The unit of $(\cc,F)$ is $(\varepsilon_\CM,\eta)$ where $\eta=\{\eta_x \co \un \to A_x\}_{x \in H}$ is a natural transformation  from $\underline{\un}\varepsilon_\CM$ to $F$.
The coproduct of $(\cc,F)$ is $(m_\CM,\Delta)$ where $\Delta=\{\Delta_{x,y} \co A_{xy} \to A_x \otimes_\kk A_y \}_{x,y \in H}$ is a natural transformation  from $Fm_\CM$ to $F \otimes F$. The counit of $(\cc,F)$ is $(u_\CM,\varepsilon)$ where~$\varepsilon$ is a natural transformation from $Fu_\CM$ to $\underline{\un}$, that is, a morphism $\varepsilon \co A_1 \to \un$ in $\sss$. The antipode of $(\cc,F)$ is $(S_\CM,S)$ where $S=\{S_x \co A_{x^{-1}} \to A_x\}_{x \in H}$ is a natural transformation from $FS_\CM$ to $F$. The associativity and unitality  of the product of $(\cc,F)$ gives that each $A_x$ is an algebra in $\sss$ with unit  $\eta_x$. The naturality  of the product and the unit of $(\cc,F)$ give that each $\phi_{x,e}$ is an algebra morphism.
The coassociativity and  counitality  of the coproduct of $(\cc,F)$ gives that $\Delta$ is coassociative and counital. 
The naturality  of the coproduct of $(\cc,F)$ amounts to Axiom~\eqref{eq-Xi-act-3}. The naturality of the counit is automatic. The fact that $(\cc,F)$ is a bialgebra in~$\ee(\sss)$ implies that each  $\Delta_{x,y}$ and $\varepsilon$ are algebra morphisms. The fact that $(S_\CM,S)$ is an invertible antipode for $(\cc,F)$ implies that the $S_x$ are invertible and satisfy the axiom of Section~\ref{sect-Hopf-H-coalgebras-antipodes}. (Note that the naturality of the antipode of $(\cc,F)$ amounts to the property of Lemma~\ref{lem-Xi-antipode} extended to Hopf $\CM$-coalgebras in $\sss$, and so can be deduced from the other axioms.) Then~$A$ is a Hopf $\CM$-coalgebra in $\sss$.

Conversely, any Hopf $\CM$-coalgebra $A=\{A_x\}_{x \in H}$ in $\sss$ with $\CM$-action $\phi$ gives rise to a Hopf algebra $(\tg_\CM,F_A)$ in $\ee(\sss)$, where the functor $F_A \co \tg_\CM \to \Mod_\kk$ is defined by
$F_A(x)=A_x$  and $F_A\bigl(x \xrightarrow{e} \CM(e)x\bigr)=\phi_{x,e}$,
with (co)product, (co)unit, and antipode derived from those of $A$ as above.
\end{proof}

\subsection{Hopf algebras in $\ff(\sss)$}\label{sect-Hopf-alg-FS}
We associate to $\sss$ another symmetric monoidal category $\ff(\sss)$ defined as follows. 

Objects of $\ff(\sss)$ are pairs $(\cc,F)$ where~$\cc$ is a small category and $F\co \cc \to \sss$ is a functor. A morphism from $(\cc,F)$ to $(\dd,G)$ is a pair~$(\Gamma,\gamma)$ where $\Gamma\co \cc \to \dd$ is a functor and $\gamma$ is a natural transformation from $F$ to $G \Gamma$. The composition, monoidal product, and symmetry of $\ff(\sss)$ are defined in a way similar to~$\ee(\sss)$. Explicitly, the composition of $(\Gamma,\gamma)\co (\cc,F) \to (\dd,G)$ with $(\Lambda,\lambda)\co (\dd,G) \to (\mathcal{K},H)$ is 
$$
(\Lambda,\lambda) \circ (\Gamma,\gamma)=( \Lambda\Gamma, \lambda_\Gamma \gamma) \quad \text{where} \quad
\lambda_\Gamma \gamma=\{\lambda_{\Gamma(X)} \circ \gamma_X\}_{X\in \mathcal{\cc}}.
$$
The identity of $(\cc,F)$ is $(1_\cc,\id_F)$.
The monoidal product of two objects $(\cc,F)$ and $(\dd,G)$ is defined by
$$
(\cc,F) \boxtimes (\dd,G)=\bigl(\cc \times \dd, F\otimes G=\otimes \circ (F \times G) \co \cc \times  \dd \to \sss\bigr).
$$
The monoidal product of a morphism $(\Gamma,\gamma)\co (\cc,F) \to (\cc',F')$ with a morphism $(\Lambda,\lambda)\co (\dd,G) \to (\dd',G')$ is defined by
$$
(\Gamma,\gamma) \boxtimes (\Lambda,\lambda)=\bigl(\Gamma \times \Lambda,\gamma \otimes \lambda=\{\gamma_X \otimes \lambda_Y\}_{(X,Y) \in \cc \times \dd} \bigr).
$$ 
The unit object of $\ff(\sss)$ is $I=(\mathbf{1},\underline{\un})$, see Section~\ref{sect-ES}. The monoidal constraints and symmetry of $\ff(\sss)$ are induced (in the obvious way) from those of $\Cat$ and $\sss$. In particular, the forgetful functor $\mathcal{U}\co\ee(\sss) \to \Cat$, defined by $\mathcal{U}(\cc,F)=\cc$ and $\mathcal{U}(\Gamma,\gamma)=\Gamma$,
is strict monoidal and symmetric.

\begin{thm}\label{thm-characterization-Hopf-Xi-algebras-in-S}
Hopf algebras in $\ff(\sss)$ are in bijective correspondence with pairs $(\CM,A)$ where $\CM$ is a crossed module and $A$ is a Hopf $\CM$-algebra in $\sss$.
\end{thm}

\begin{proof}
Recall that $\cc^\opp$ denotes the category opposite to a category $\cc$. The opposite of a functor $F\co \cc \to \dd$ is the functor $F^\opp\co \cc^\opp \to \dd^\opp$ acting as $F$ on objects and morphisms. The opposite of a natural transformation $\alpha$ from a functor $F \co \cc \to \dd$ to a functor $G \co \cc \to \dd$ is the natural transformation $\alpha^\opp$ from $F^\opp \co \cc^\opp \to \dd^\opp$ to  $G^\opp \co \cc^\opp \to \dd^\opp$ defined by $\alpha^\opp_X=\alpha_X$ for all $X \in \Ob(\cc^\opp)=\Ob(\cc)$.

Observe that the assignments $(\cc,F) \mapsto (\cc^{\opp},F^{\opp})$ and $(\Gamma,\gamma) \mapsto (\Gamma^{\opp},\gamma^{\opp})$ induce a symmetric strict monoidal isomorphism $\ff(\sss) \cong (\ee(\sss^\opp))^\opp$. Consequently, Hopf algebras in $\ff(\sss)$ are in bijective correspondence with Hopf algebras in $(\ee(\sss^\opp))^\opp$, and so  with Hopf algebras in $\ee(\sss^\opp)$ (by exchanging product and coproduct), and so with pairs $(\CM,A)$ where $\CM$ is a crossed module and $A$ is a Hopf $\CM$-coalgebra in $\sss^\opp$ (by Theorem~\ref{thm-characterization-Hopf-Xi-coalgebras-in-S}). We conclude using that  Hopf $\CM$-coalgebras in~$\sss^\opp$ bijectively correspond to Hopf $\CM$-algebras in the opposite category $\sss$.
\end{proof}

\subsection{The case of Hopf crossed module-(co)algebras over $\kk$}
Applying Theorem~\ref{thm-characterization-Hopf-Xi-coalgebras-in-S} to the symmetric monoidal category $\Mod_\kk$ of \kt modules and \kt linear homomorphisms, we obtain the following  characterization Hopf crossed module-coalgebras over $\kk$ as Hopf algebras:

\begin{cor}\label{cor-characterization-Hopf-Xi-coalgebras}
Hopf algebras in $\ee(\Mod_\kk)$ are in bijective correspondence with pairs $(\CM,A)$ where $\CM$ is a crossed module and $A$ is a Hopf $\CM$-coalgebra over $\kk$.
\end{cor}

Similarly, applying Theorem~\ref{thm-characterization-Hopf-Xi-algebras-in-S} to $\Mod_\kk$, we obtain the following  characterization Hopf crossed module-algebras over $\kk$ as Hopf algebras:
\begin{cor}\label{cor-characterization-Hopf-Xi-algebras}
Hopf algebras in $\ff(\Mod_\kk)$ are in bijective correspondence with pairs $(\CM,A)$ where $\CM$ is a crossed module and $A$ is a Hopf $\CM$-algebra over $\kk$.
\end{cor}

Recall that if $H$ is a group, then the trivial map $1\to H$ is a crossed module and the notion of a Hopf $(1\to H)$-(co)algebra over $\kk$ corresponds to that of a Hopf $H$-(co)algebra over~$\kk$. 
Let $\ee_d(\Mod_\kk)$ and $\ff_d(\Mod_\kk)$  be the symmetric monoidal full subcategories of $\ee(\Mod_\kk)$ and $\ff(\Mod_\kk)$ consisting of the objects $(\cc,F)$ with~$\cc$ a discrete small category (that is, a set).
By restricting the correspondences of Corollaries~\ref{cor-characterization-Hopf-Xi-coalgebras} and~\ref{cor-characterization-Hopf-Xi-algebras}  to these subcategories, we recover the following  
characterizations given in~\cite{CD}: Hopf algebras in $\ee_d(\Mod_\kk)$ (respectively, in $\ff_d(\Mod_\kk)$) are in bijective correspondence
with pairs $(H,A)$ where $H$ is a group and $A$ is a Hopf $H$-coalgebra  over~$\kk$ (respectively, a Hopf $H$-algebra over~$\kk$).

\subsection*{Acknowledgements}
This work was supported in part by the FNS-ANR grant OCHoTop (ANR-18-CE93-0002) and the Labex CEMPI  (ANR-11-LABX-0007-01).

\end{document}